\documentclass[
 jmp,
 amsmath,amssymb,
 reprint
]{revtex4-1}

\usepackage{graphicx}
\usepackage{dcolumn}
\usepackage{bm,amsthm}

\usepackage[utf8]{inputenc}
\usepackage[T1]{fontenc}
\usepackage{mathptmx}
\usepackage{etoolbox}

\makeatletter
\def\@email#1#2{%
 \endgroup
 \patchcmd{\titleblock@produce}
  {\frontmatter@RRAPformat}
  {\frontmatter@RRAPformat{\produce@RRAP{*#1\href{mailto:#2}{#2}}}\frontmatter@RRAPformat}
  {}{}
}%
\makeatother

\usepackage[dvips]{epsfig}

\newtheorem{thm}{Theorem}[section]
\newtheorem{lemma}[thm]{Lemma}

\newtheorem{prop}[thm]{Proposition}

\def\C{{\mathbb C}}
\def\del{\delta}
\def\sign{{\rm sign}}
\def\less{\lesssim}
\def\la{\langle}
\def\ran{\rangle}
\def\ra{\rangle}

\def\Re{{\,\rm Re\,}}
\def\Im{{\,\rm Im\,}}

\def\Ai{{\rm Ai}}
\def\Bi{{\rm Bi}}
\def\eps{{\varepsilon}}
\def\les{\lesssim}

\def\calM{{\mathcal M}}
\def\Rd{{\mathbb R}}
\def\nn{\nonumber}
\def\calH{\mathcal H}
\def\del{\partial}

\def\prefac{\frac{i}{\hbar}}

\def\snabla{ {\slash\!\!\! \nabla} }
\def\calE{{\mathcal E}}
\def\Ran{\mathrm{Ran}}
\def\mc{\mathcal}
\def\Z{{\mathbb Z}}
\def\snabla{ {\slash\!\!\! \nabla} }

\def\Id{\mathrm{Id}}
\def\calB{\mathcal{B}}
\def\calM{\mathcal{M}}
\def\calA{\mathcal{A}}
\newcommand{\oast}{\circledast}
\def\calD{\mathcal{D}}

\begin{document}

\title[Dispersive decay of  waves]{On pointwise decay of waves}

\author{W. Schlag}
\affiliation{Yale University, Department of Mathematics,
10 Hillhouse Avenue, New Haven, CT 06511, U.S.A.}
\email{wilhelm.schlag@yale.edu}

\begin{abstract} 
This note introduces some of the basic mechanisms relating the behavior of the spectral measure of 
Schr\"odinger operators near zero energy 
to the long-term decay and dispersion of the associated Schr\"odinger and wave evolutions. 
These principles are illustrated by means of  the author's work  on decay of Schr\"odinger and wave equations
under various types of perturbations including those of the underlying metric. In particular, we
consider  local decay of solutions to the
linear Schr\"odinger and wave equations on curved backgrounds which exhibit trapping. A particular application
 are waves on a Schwarzschild black hole space-time. We elaborate on Price's
law of local decay which accelerates with the angular momentum, which has recently been settled by 
Hintz, also in the much more difficult Kerr black hole setting. 
While the author's work on the same topic was conducted ten years ago, the global semiclassical  representation techniques  developed there have recently been applied by Krieger, 
Miao, and the author to the nonlinear problem of stability of blowup solutions to critical wave maps under non-equivariant perturbations. 
\end{abstract}

\date{\today}

\maketitle

\section{Introduction}\label{sec:Prelim}

This note mainly serves as an introduction to  the  techniques used in the papers~\cite{DSS1}, \cite{DSS2},
 which are concerned with the local decay of waves on a Schwarzschild background. The decay estimates are obtained by separation of variables and 
 the  analysis of the flow for each angular momentum $\ell$ in~\cite{DSS1}. By means of a  semiclassical WKB  analysis in the parameter $\hbar:=\ell^{-1}$ carried out by means of a global Liouville-Green transform, as well as semiclassical Mourre theory at energies near the top of the barrier, \cite{DSS2} 
sums up  these fibered  estimates 
over all angular momenta incurring the loss of finitely many angular derivatives.  
Note that \cite{DSS1}, \cite{DSS2} are not entirely self-contained and rely in part on  \cite{SSS1}, \cite{SSS2}, 
\cite{DS}, \cite{CSST}, \cite{CDST}.  As shown in these references, the Schr\"odinger flow can be analyzed analogously. The original motivation for \cite{SSS1, SSS2}
was to study the long-term dispersive behavior of solutions to 
Schr\"odinger and wave equations on specific non-compact manifolds exhibiting closed geodesics, such as the hyperboloid of one sheet. 
In analogy with the unique periodic geodesic on such a hyperboloid, which is exponentially unstable, the surface of closed geodesics
around a Schwarzschild black hole is known as {\em photon sphere} and corresponds to the collection of all periodic light rays. The photon sphere is also unstable. 

Recently, in joint work with Krieger and Miao~\cite{KMS}, the  semiclassical techniques leading to a precise  representation of the resolvent and the spectral measure  for all energies and all small $\hbar$ developed in~\cite{CSST,CDST}, played a crucial role in a nonlinear asymptotic stability question of blowup solutions to energy critical wave maps into the $2$-sphere. In stark contrast to the linear case, modes of fixed frequencies interact through the nonlinearities. Controlling these interaction naturally leads to a paradifferential calculus involving several simultaneous semi-classical parameters. 
The nonlinear work~\cite{KMS} served as the main motivation for writing this note, which should not be mistaken for a general review. Numerous references are missing,  which touch in one way or another on the ensuing discussion. A survey of dispersive decay of Schr\"odinger, wave, and Klein-Gordon evolutions involving electric, magnetic, and metric perturbations, including the semi-classical and gravitational literature, would require many hundreds of citations.  The scope and purpose of this communication is much more limited. For example, magnetic and time-dependent potentials  are not discussed in detail. 

The author's investigations in this area where largely motivated by Bourgain's book~\cite{Bour} which states at the end of page~27:  {\it On the other hand, it would be most interesting to prove that analogue of (1.99) in low dimensions $d=1,2$. This is certainly a project of independent importance.} Here (1.99) refers to the pointwise decay of the Schr\"odinger evolution proved by Journ\'e, Soffer and Sogge~\cite{JSS}, see the following section.  

\section{Lower order perturbations}\label{sec:low}

The free Schr\"odinger evolution $\psi(t)=e^{-it\Delta}\psi_0$ in $\Rd^{d+1}_{t,x}$ satisfies the basic estimates
\begin{align}
\|\psi(t)\|_{H^s} &= \|\psi_0\|_{H^s} \label{eq:schL2}\\
\|\psi(t)\|_\infty &\le Ct^{-\frac{d}{2}} \|\psi_0\|_1 \label{eq:sch_disp}
\end{align}
as can be seen from the representation
\begin{align*}
\psi(t,x) &= (2\pi)^{-d} \int_{\Rd^d} e^{i(t |\xi|^2+x\cdot\xi)} \hat{f}(\xi)\, d\xi \\
&= c(d) t^{-\frac{d}{2}} \int_{\Rd^d} e^{i\frac{|x-y|^2}{4t}} f(y)\, dy
\end{align*}
respectively. For the wave equation $\Box u=\del_t^2 u - \Delta u=0$ in $d+1$ dimensions one has constancy of the
energy
\begin{equation}\label{eq:wave_energy}
\calE(u) = \|\nabla u\|_2^2 + \|\del_t u\|_2^2
\end{equation}
as well as the dispersive decay
\begin{equation}\label{eq:wave_disp}
\|u(t)\|_{\infty} \les t^{-\frac{d-1}{2}} (\|u(0)\|_{\dot B_{1,1}^{\frac{d+1}{2}}} + \|\partial_t u(0)\|_{\dot B_{1,1}^{\frac{d-1}{2}} })
\end{equation}
where $\dot B_{1,1}^\alpha$ stands for the usual Besov space: $\|f\|_{\dot B^\alpha_{1,1}}=\sum_{j\in\Z} 2^{\alpha j} \|P_j f\|_1$ where $P_j$ is
the Littlewood-Paley projection onto frequencies of size~$2^j$.
In odd spatial dimensions one can improve the right-hand side to
\[
\|u(0)\|_{\dot W^{\frac{d+1}{2},1}}  + \|\partial_t u(0)\|_{\dot W^{\frac{d-1}{2},1}}
\]
where $\dot W^{\alpha,p}$ stands for the homogeneous Sobolev spaces.
To obtain~\eqref{eq:wave_disp}, one considers a fixed frequency shell $\{|\xi|\sim 2^j\}$ and rescales
to~$j=0$. Then
\[
e^{it\sqrt{-\Delta}}P_0  f(x) = \int_{\Rd^{2d}} e^{i((x-y)\cdot\xi+t|\xi|)} \chi(\xi)\,d\xi \, f(y)\, dy
\]
where $\chi$ is a cut-off function corresponding to~$P_0$. Passing to polar coordinates and applying
stationary phase to integrals over spheres then yields the desired $t^{-\frac{d-1}{2}}$ decay.

While \eqref{eq:schL2} and~\eqref{eq:wave_energy} are a result of the time-translation
invariance of the underlying Lagrangians (via Noether's theorem) and therefore robust
under perturbations that preserve this symmetry, \eqref{eq:sch_disp} and~\eqref{eq:wave_disp} follow
from the form of the fundamental solutions and are therefore less stable.
In fact, much effort has been devoted to deriving similar dispersive estimates for perturbations of the free
Schr\"odinger and wave equations in the past thirty years.  The starting point in these investigations was
to consider {\em local} decay estimates which are quite different from the global ones as
in~\eqref{eq:sch_disp} and~\eqref{eq:wave_disp} (as we shall see below). {\em Local} here refers to the fact that the decay is measured only in weighted spaces rather than
in a uniform sense.  

\subsection{Local decay for $-\Delta+V$}\label{subsec:local}

\subsubsection{The Schr\"odinger evolution}
In~\cite{JK} Jensen and Kato showed that for $H=-\Delta+V$ in the three-dimensional case, with real-valued $V$ which  is bounded
and decays at a sufficient polynomial rate one has the local decay
\begin{equation}
\label{eq:JK}
\| \la x\ra^{-\sigma} e^{itH} P_c f\|_{L^2(\Rd^3)} \les \la t\ra^{-\frac32} \|\la y\ran^\sigma f\|_{L^2(\Rd^3)}
\end{equation}
for some $\sigma>0$ and with $P_c=\chi_{(0,\infty)}(H)$ the projection onto the continuous spectrum.
Moreover, one needs to assume that zero energy is neither an eigenvalue nor a resonance of~$H$ (which is also referred to as zero
energy being regular, the other case being singular).

This latter property refers to the validity of the resolvent estimate 
\begin{eqnarray}
\sup_{\Im z>0} \|\la x\ra^{-\sigma} (-\Delta +V+z)^{-1} \la x\rangle^{-\sigma} \|_{2\to2} <\infty
\end{eqnarray}
with $\sigma>0$ sufficiently large. Alternatively, it is the same as the nonexistence of $f\not\equiv0$ with
\begin{eqnarray}
\label{eq:res}
Hf=0,\quad f\in\bigcap_{\eps>0} L^{2,-\frac12-\eps}(\Rd^3)
\end{eqnarray}
It was already observed by Rauch~\cite{Rauch} for exponentially decaying potentials,  that a zero energy resonance or eigenvalue, i.e.,
in the case when~\eqref{eq:res} admits a nontrivial solution, destroys the dispersive estimate. More specifically, one loses
one power of~$t$ in the decay law in that case.

To see the relevance of zero energy resonances, we expand the resolvent for $z\to0$ in $\Im z>0$ as follows:
\begin{align}\label{eq:laurent}
R(z) &:=(-\Delta+V+z)^{-1} \\
&= z^{-1} B_{-1} + z^{-\frac12} B_{-\frac12} + B_0 + z^{\frac12} B_{\frac12} + \rho(z)\nn
\end{align}
where $B_{-1},\ldots, B_1$ are bounded in weighted $L^2(\Rd^3)$-spaces, and with \[\| \la x\rangle^{-\sigma} \rho(z) f\|_2 \les |z|\|\la x\rangle^{\sigma} f\|_2\]
for small $z$.
Clearly, $B_{-1}$ is the orthogonal projection onto the zero eigenspace, and zero energy is regular 
for $H$ iff $B_{-1}=B_{-\frac12}=0$. In general, $B_{-1}, B_{-\frac12}$ are of finite rank.
As an example, consider the case $V=0$ in three dimensions for which one has (with $z=-\zeta^2)$
\[
(-\Delta-\zeta^2)^{-1}(x,y) = \frac{e^{i\zeta|x-y|}}{4\pi |x-y|},\quad \Im \zeta>0
\]
and the Laurent expansion \eqref{eq:laurent} is now obtained by Taylor expanding the exponential on the right-hand side.
It follows that zero energy is neither an eigenvalue nor a resonance in that case. In contrast, the one-dimensional case satisfies
\[
(-\Delta-\zeta^2)^{-1}(x,y) = \frac{e^{i\zeta|x-y|}}{2i\zeta},\quad \Im \zeta>0
\]
and zero is a resonance (but not an eigenvalue). We  used here that~\eqref{eq:laurent} remains correct in all {\em odd}
dimensions, whereas in even dimensions a logarithm appears. Indeed, the free resolvent in $d$-dimensions satisfies
\begin{equation}\label{eq:resolvent}
(-\Delta-\zeta^2)^{-1}(x,y) = c_d\, \zeta^{\frac{d-2}{2}} |x-y|^{-\frac{d-2}{2}} H_{\frac{d-2}{2}}^{+}(\zeta|x-y|)
\end{equation}
and the Hankel functions of integer order exhibit a logarithmic branch point at zero.

To pass to estimates on the evolution one now uses the Laplace transform (as in the Hille-Yosida theorem) to conclude that
\begin{equation}\label{eq:laplace}
e^{itH}P_c = \frac{1}{2\pi } \int_{p_0-i\infty}^{p_0+i\infty} e^{tp} R(ip) P_c\, dp
\end{equation}
where $p_0>0$ is arbitrary. Assuming for simplicity that $V$ is compactly supported it follows from the resolvent identity
that the Green function $R(ip)(x,y)$
admits a meromorphic continuation to the left-half plane.
\begin{figure*}[ht]
\begin{center}
 \centerline{\hbox{\vbox{ \epsfxsize= 8.0 truecm \epsfysize=5.5
 truecm \epsfbox {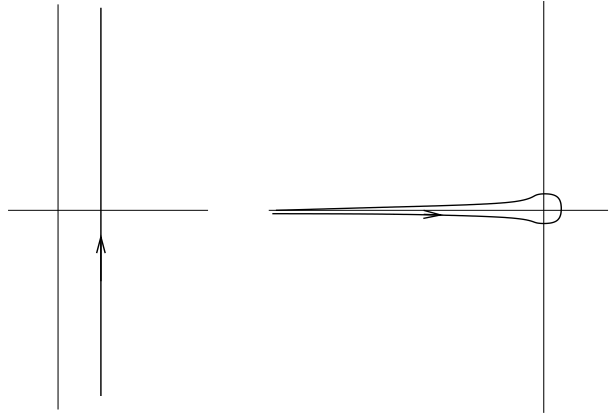}}}} \caption{Deforming the contour}
\end{center}
\end{figure*}
One now deforms the contour in~\eqref{eq:laplace} as shown
in Figure~1. The finitely many residues $\{\zeta_j\}$ of the resolvent in the left-half plane (which lie in $\C\setminus(-\infty,0]$) contribute
the exponentially decaying expression
\[
\sum_{\zeta_j} e^{\zeta_j t}  P_{\phi_j}
\]
where $P_{\phi_j}$ is the projection onto the resonant states corresponding to the complex resonance at $\zeta_j$ (the resonant states are commonly referred to as {\it meta-stable states} or {\it  quasinormal modes}). The more slowly decaying tail is a result of the branching of the resolvent at $p=0$. More specifically,
it can be read off from~\eqref{eq:laurent} via the following standard result which is known as Watson's lemma (the notation $\sim$ refers to 
asymptotic expansions in the sense of Poincar\'e).

\begin{lemma}
\label{lem:watson}
Let $f$ be a complex-valued function of a real variable $x$ such that
\begin{itemize}
\item $ f$ is continuous on $(0,\infty)$
\item  $f(x)\sim \sum_{n=0}^\infty a_n\, x^{\lambda_n-1} \qquad\text{\ \ as\ \ } x\to0+$ with $0<\lambda_0<\lambda_1<\ldots$.
\item $f(x)=O(e^{cx})$ as $x\to\infty$ for some $c>0$
\end{itemize}
This condition can be removed since Watson's lemma is really local on some interval $(0,x_0)$, 
but we choose to state it in this global form.
Then  for every small $\delta>0$ one has
\[
\int_0^\infty e^{-xp} f(x)\, dx \sim \sum_{n=0}^\infty \frac{a_n}{p^{\lambda_n}} \Gamma(\lambda_n)\]
as $|p|\to\infty$ in $|\arg(p)|\le \frac{\pi}{2}-\delta$.
\end{lemma}

Therefore,  if $B_{-\frac12}\ne0$ in~\eqref{eq:laurent}, then one obtains $t^{-\frac12}$ {\em local} decay,
whereas otherwise the rate is $t^{-\frac32}$ which is the same as in~\eqref{eq:sch_disp}.  Evidently,
the global (i.e., $L^\infty$) decay can never be faster than the local one --- whence the need to exclude zero energy
resonance and eigenvalues to preserve~\eqref{eq:sch_disp}. We remark that one can have $B_{-\frac12}\ne0$ even in case
the only solutions to~\eqref{eq:res} are in~$L^2$ (in other words, if zero energy is an eigenvalue but not a resonance).
This implies that $t^{-\frac32}$ does not result from applying~$P_c$ to the evolution even when zero is not a resonance
but only an eigenvalue.

Starting from the spectral representation
\begin{equation}
\label{eq:evolution}
e^{itH}P_c =  \int_0^\infty e^{it\lambda}\, E(d\lambda)
\end{equation}
instead of \eqref{eq:laplace} with the spectral measure
\[
E(d\lambda)=\frac{1}{2\pi i}[R(\lambda+i0)-R(\lambda-i0)]P_c\,d\lambda
\]
Jensen and Kato derive local decay estimates but under much less severe restrictions on the decay of $V$
and also on the notion of {\em locality} in the decay estimate. However, it is clear from~\eqref{eq:evolution} that
the main issue here is once again the contributions from $\lambda=0$ coming from~\eqref{eq:laurent}.
Indeed, for energies $\lambda>\lambda_0>0$ where $\lambda_0>0$ is arbitrary but fixed, one has the so-called {\em limiting absorption}
resolvent bounds
\[
\sup_{\lambda>\lambda_0} \big\| \la \cdot\rangle^{-\sigma}\del_\lambda^k R(\lambda\pm i0)  \la \cdot\rangle^{-\sigma} \big\| <\infty
\]
for all $0\le k\le k_0$ and with $\sigma>0$ depending on~$k$ (the value of $k_0$ here depends on the decay of~$V$).  These bounds allow one to integrate by parts in~\eqref{eq:evolution} in the range $\lambda>\lambda_0$ which leads to arbitrary decay in time.

The most general results on local decay for the Schr\"odinger evolution
 were obtained by Murata~\cite{Mur}. He derived expansions in time for evolutions $e^{itH}$ in {\em all dimensions} and with
elliptic  $H=-p(D)+V$ where $V$ is a compact operator in suitable weighted Sobolev spaces. As a general rule, the coefficients
in these expansions corresponding to nongeneric threshold behavior (i.e., slow decay resulting from threshold eigenvalues or resonances)
are {\em finite rank} operators which can be computed in terms of the eigenfunctions and resonant states. As an example, the
one-dimensional free evolution satisfies
\begin{align*}
e^{-it\partial_x^2} f(x) &= ct^{-\frac12} \int f(y)\, dy + \rho(t) f(x), \\
\| \la \cdot\rangle^{-\sigma} \rho(t)f\|_2& \les t^{-\frac32}\|\la \cdot \rangle^\sigma f\|_2
\end{align*}
The appearance of the projection $f\mapsto \int f(y)\, dy$ onto the constant functions is natural in view of the fact that
the resonant function of $-\partial_x^2$ at zero energy is $f\equiv1$.
This also shows that one should expect $t^{-\frac32}$ local decay for one-dimensional operators without zero energy resonance (note that, however, the {\em global} decay as in~\eqref{eq:sch_disp} is never faster than~$t^{-\frac12}$ if $d=1$), at least assuming sufficient decay of~$V$.
This is indeed the case, see~\cite{Mur}. In two dimensions, Murata obtained the faster local $L^2(\Rd^2)$ decay  $t^{-1}\log^{-2} t$ for
operators without resonance. Erdogan, Green~\cite{EG2} established the more difficult sharp weighted $L^1\to L^\infty$ version of these global bounds in $\Rd^2$ assuming that $0$ energy is regular. These faster local decays (as compared to the global $L^\infty$ decay) play a crucial role in certain
applications to nonlinear stability results, see   Buslaev, Perelman~\cite{BP1}, Krieger, Schlag~\cite{KS1} for the
 one-dimensional case,  and Kirr, Zarnescu~\cite{KZ2} for examples of two-dimensional applications. Loosely speaking, the point here is that in contrast
 to the global decay rates these faster non-resonant local rates are {\em integrable in time} which allows one to close certain bootstrap
 arguments involving the Duhamel formula.

\subsubsection{The wave evolution}
Similar considerations apply to the wave equation. Indeed, let $\Box u=0$, with $(u(0),\partial_t u(0))=(0,g)$ (initial data $(f,0)$ are
then handled by differentiating in time). Then instead
of~\eqref{eq:laplace} one has
\begin{equation}
\label{eq:wave}
u(t)=\frac{\sin(t\sqrt{H})}{\sqrt{H}} P_c g = \frac{1}{2\pi i} \int_{p_0-i\infty}^{p_0+i\infty} e^{t p} R(p^2) P_c\,g dp
\end{equation}
where $p_0>0$. In contrast to the Schr\"odinger case the resolvent $R(p^2)$ in {\em odd dimensions} is now analytic around $p=0$ (assuming
that there is no zero energy resonance or eigenvalue) which results in {\em arbitrary local decay} of $u(t)$. More precisely,
if $V$ decays exponentially, thus allowing for analytic continuation of the Green function to the left-half plane, one obtains exponential
decay in time relative to weighted $L^2$ in space. 
This is of course a consequence of the {\em sharp Huyghens principle} in odd dimensions which states that the fundamental solution
of the free wave equation is localized to a sphere with radius given by the time. We see from this informal discussion that this principle
is robust under perturbations (at least in the sense that the perturbed wave $u(t)$ will decay very rapidly at distances $\ll t$ from the origin, which of course is far from being able to describe the fundamental solution).  Note the stark contrast between
the strong local decay of the wave equation as compared to the specific global decay given by~\eqref{eq:wave_disp}.

On the other hand, in even dimensions the resolvent will exhibit a $\log p$ singularity, see~\eqref{eq:resolvent}. Due to this branching
of the resolvent at $p=0$, Watson's lemma implies an explicit power law depending on the dimension governing the tail of the wave near the origin. This is in agreement with the fact that there is no sharp Huyghens principle in even dimensions.

To summarize this section, one sees that the  {\em local decay} for both the Schr\"odinger and  the wave equation is entirely determined by the singularity (often  but
not necessarily by branching) of the resolvent $(-\Delta+V+z)^{-1}$ at $p=0$ where $z=-ip$ in the former case, and $z=p^2$ in the latter case, respectively.

\subsection{Global decay for $-\Delta+V$}\label{sec:global}
\subsubsection{The Schr\"odinger evolution}
The first result which proved \eqref{eq:sch_disp} for $H=-\Delta+V$ in dimensions $d\ge3$ was obtained by Journ\'e, Soffer, and Sogge~\cite{JSS}. Following unpublished work by Ginibre, we now give a short proof of a simpler estimate, namely
\begin{equation}
\label{eq:gin}
\|e^{itH} P_c f\|_{L^\infty+L^2(\Rd^d)} \les \la t\rangle^{-\frac{d}{2}} \|f\|_{L^1\cap L^2(\Rd^d)}
\end{equation}
assuming that $V$ has sufficient decay and that $H$ has no zero energy eigenvalue or resonance. The logic here is that
the Duhamel formula allows one to upgrade local decay to global one. More precisely, if
\[
\| \la x\ra^{-\sigma} e^{itH} P_c f\|_{L^2(\Rd^d)} \les \la t\ra^{-\frac{d}{2}} \|\la y\ran^\sigma f\|_{L^2(\Rd^d)}
\]
and if $V$ decays sufficiently fast, then the same estimate holds without weights in the sense of~\eqref{eq:gin} (provided $d>2$).
More precisely, applying the Duhamel formula twice yields
\begin{align*}
e^{itH}P_c &= e^{-it\Delta} P_c + i\int_0^t e^{-i(t-s)\Delta} V e^{isH} P_c\, ds \\
&= e^{-it\Delta} P_c + i\int_0^t e^{-i(t-s)\Delta} V P_c e^{-is\Delta} \, ds\\
&\qquad  + \int_0^t \int_0^s e^{-i(t-s)\Delta} V e^{i(s-s')H} P_c V e^{-is'\Delta} \, ds'\, ds
\end{align*}
Applying the local decay for $e^{isH}$ from the previous section (with $|V|^{\frac12}$ acting as weight, say) as well as the bound
\[
\|e^{-it\Delta} f\|_{L^2+L^\infty(\Rd^d)} \les \la t\rangle^{-\frac{d}{2}} \|f\|_{L^1\cap L^2(\Rd^d)}
\]
to this expression yields for  $\|f\|_{L^1\cap L^2(\Rd^d)}=1$
\begin{align*}
& \|e^{itH}P_c f\|_{L^\infty+L^2(\Rd^d)} \les \la t\rangle^{-\frac{d}{2}}    + \int_0^t \la t-s\rangle^{-\frac{d}{2}}  \la s\rangle^{-\frac{d}{2}}\, ds\\
& + \int_0^t \int_0^s \la t-s\rangle^{-\frac{d}{2}} \la s-s'\rangle^{-\frac{d}{2}}   \la s'\rangle^{-\frac{d}{2}}  \, ds'\, ds
\les \la t\rangle^{-\frac{d}{2}}
\end{align*}
as claimed provided $d\ge3$. The main gist of~\cite{JSS} is now to remove the $L^2$-piece from this argument. This is subtle, as the free estimate involved $(t-s)^{-\frac{d}{2}}$ which is not integrable at $s=t$. To overcome this difficulty, Journ\'e,
Soffer, and Sogge used the bound
\[
\sup_{1\le p\le \infty}\| e^{-it\Delta} V e^{it\Delta} \|_{p\to p} \le \|\hat{V}\|_1
\]
The point here is that the left-hand side for $V=e^{ix\eta}$ is a translation operator composed with a unimodular factor and therefore
$L^p$ bounded. 

Rodnianski and the author~\cite{RS}  proved that for all $t>0$ 
\begin{equation}\label{eq:32tdec}
\|e^{itH}f\|_{L^\infty(\Rd^3)}\le C(V) t^{-\frac32}\|f\|_{L^1(\Rd^3)}
\end{equation} 
assuming that
\begin{equation}\label{eq:kato_norm}
\sup_{x\in\Rd^3} \int_{\Rd^3} \frac{|V(y)|}{|x-y|}\, dy < 4\pi
\end{equation}
as well as that the so-called Rollnick norm of~$V$ is less than $4\pi$. 
The left-hand side in~\eqref{eq:kato_norm} is commonly referred to as the {\em Kato norm} $\|\cdot\|_{K}$.
The Rollnick condition precludes any spectral problems such as eigenvalues and a zero energy singularity.
The approach of~\cite{RS} to the pointwise bounds is based on an expansion into an infinite Born series followed by term-wise estimation
of the resulting kernels. The smallness condition on $V$ guarantees convergence. 

Remarkably, Beceanu and Goldberg~\cite{BecGol} were able to show that the  finiteness of the Kato norm alone suffices. More precisely, they showed that~\eqref{eq:32tdec} holds for $e^{itH}P_c$ in three dimensions assuming~\eqref{eq:kato_norm} with $4\pi$ replaced by~$\infty$ and  that there are no imbedded eigenvalues and resonances in the continuous spectrum. 
They accomplished this by means of Beceanu's Wiener algebra techniques, see~\cite{Bec1}.  
Recall that Wiener's classical theorem  states that for any $f\in L^1(\Rd)$ the equation $(\delta_0+f)\ast (\delta_0+g)=\delta_0$ has a (unique) solution with $g\in L^1(\Rd)$ if and only if $1+\hat{f}\ne 0$ on $\Rd$.  The relevance of this to the decay of solutions to 
\[
(i\partial_t -\Delta +V)\psi =F,\quad \psi(0)=\psi_0
\]
 can be seen as follows: let $V_1V_2=V$, $|V_1|=|V_2|$ and set
\[
(T_{V_2, V_1} F)(t) = \int_{0}^t V_2 e^{i(t-s)H_0} V_1 F(s)\, ds
\]
with  $H_0=-\Delta$. 
Then, on the one hand, one has 
\[
V_2 \psi(t) = (\delta_0\Id -i T_{V_2, V_1})^{-1} V_2\big(e^{itH_0} \psi_0 - i\int_0^t e^{i(t-s)H_0} F(s)\, ds \big) 
\]
which is to be interpreted in the convolution algebra $\calB(L^2(\Rd^3), \calM_t L^2(\Rd^3))$  where $\calM_t$ are the complex measures on the line. 
On the other hand, 
$
\widehat{T_{V_2, V_1} }(\lambda) = iV_2 R_0^{-}(\lambda) V_1$ with $ R_0^{-}(\lambda) = (H_0 - (\lambda-i0))^{-1}$. 
Hence, the invertibility of $\delta_0\Id-i T_{V_2, V_1}$ in $\calB(L^2(\Rd^3), \calM_t L^2(\Rd^3))$ is the same as the pointwise invertibility of the Birman-Schwinger operator
$\Id +V_2 R_0^{-}(\lambda) V_1$. This equivalence is delicate, and requires $V\in L^{\frac32,1}(\Rd^3)$ the Lorentz space, whence $V_1, V_2\in L^{3,2}(\Rd^3)$,  and also the Keel-Tao Strichartz endpoint~\cite{KT}. For the abstract Wiener theorem in this context, see \cite[Theorem~1.1]{Bec1} and \cite[Theorem 3]{BecGol}.

An alternative  and very general approach to proving $L^p$ bounds on both wave and Schr\"odinger evolutions was found by Yajima~\cite{Y1}, \cite{Y2} who proved $L^p$ boundedness of the
wave operators, with the limit being taken in the strong $L^2$-sense,
\begin{equation}\label{eq:moller}
W = \lim_{t\to\infty} e^{-itH} e^{-it\Delta}
\end{equation}
for all $1\le p\le\infty$ and $d\ge3$.  The fact that
these operators exist and are isometries $L^2\to \Ran(P_c(H))$ is a classical fact, see  Kato~\cite{Kato}.
They intertwine the free evolution with that of~$H$ in the sense that (with $H_0=-\Delta$)
\[
f(H)P_c(H) = W f(H_0) W^*
\]
for any Borel function $f$ on~$\Rd$. In particular, $e^{itH}P_c(H)=  W e^{itH_0}  W^*$ and~\eqref{eq:sch_disp} therefore implies the bound
\[
\| e^{itH} P_c f\|_\infty \le Ct^{-\frac{d}{2}} \|f\|_1
\]
whenever $W:L^\infty\to L^\infty$, $W^*:L^1\to L^1$.  Yajima obtains similar results on $W^{k,p}$ assuming more regularity on~$V$
(the amount of regularity depends on~$k$).
In view of our discussions of the role of zero energy resonances for local decay, it follows
that Yajima's result~\cite{Y1} can only hold under the assumption that zero energy is neither a resonance nor an eigenvalue.
In three dimensions~\cite{Y1} requires $|V(x)|\les \la x\rangle^{-\sigma}$ with $\sigma>5$ and therefore improves on~\cite{JSS}.

Yajima derives his $L^p$ bounds by means of a finite Born series expansion with a remainder term involving the perturbed resolvent. 
In case of small potentials, one can sum up the infinite Born expansion leading to more precise results in terms of conditions on~$V$. 
In view of the preceding discussion of Wiener theorems as a means of summing divergent series,  it is natural to ask if Yajima's theorem could be approached
by means of a suitable Wiener algebra. 
Beceanu and the author~\cite{BecSch} carried this out and proved that  the wave operators given by~\eqref{eq:moller} in $\Rd^3$ are superpositions of reflections and translations. 
In fact, 
 assuming that $|V(x)|\les C\langle x\rangle^{-\frac52-\epsilon}
$, and that $0$ energy is neither an eigenvalue nor a resonance, 
they showed that there exists  $g(x, y, \omega) \in L^1_\omega \mc M_y L^\infty_x$ (with $\mc M_y$ being finite Borel measures in $y$), i.e., 
\[
  \int_{\mathbb{S}^2} \|g(x, dy, \omega)\|_{\mathcal{M}_y L^{\infty}_x} \, d \omega < \infty
\]
such that for $f \in L^2(\Rd^3)$ one has the representation formula for the wave operator
\[
(W f)(x) = f(x) + \int_{\mathbb{S}^2} \int_{\Rd^3} g(x, dy, \omega) f(S_\omega x - y)  \, d \omega.
\]
where $S_{\omega}x=x-2(x\cdot\omega)\omega$ is a reflection.  
This of course implies that $W:X\to X$ is bounded for any function space $X$ on $\Rd^3$ with a norm which is invariant under translations and reflections. 
The proof of this representation formula in~\cite{BecSch} is not entirely straightforward. On the one hand, the algebra to which the Wiener theorem is applied
is somewhat delicate and requires casting the finite order Born series terms in Yajima's work~\cite{Y1} (which involve only finitely many potentials and free resolvents) in 
some iterative algebraic framework. In other words, one needs to find the correct algebra $\calA$ and composition law $\oast$ as well as operator $T$ to write the third Born term, say, in the form $T\oast T\oast T$ in~$\calA$. Furthermore, the classical scattering theory
based on weighted $L^2$ spaces does not suffice and it is necessary to invoke the the  author's work with Ionescu~\cite{IS}, which revisits the classical Agmon-Kato-Kuroda theorem 
in the context of Fourier restriction and the Stein-Tomas theorem, as well as the
Keel-Tao endpoint~\cite{KT}. This in turn relies on the Carleman theorems and absence of imbedded eigenvalues obtained in~\cite{IJ}. 
It is not known whether a structure theorem  holds under a scaling invariant assumption on $V$,
 see however~\cite{BecSch2} for such a result, albeit   involving small scaling-invariant potentials. 
 
In higher dimensions, it turns out that one  needs to assume some regularity of~$V$ in order for the expected $L^1(\Rd^d)\to L^\infty(\Rd^d)$ bounds to hold. Indeed, Goldberg, Visan~\cite{GVis} show
that the dispersive bound can fail in dimensions $d>3$ for potentials that belong to the class $C^{\frac{d-3}{2}-}(\Rd^d)$.
The logic here is that the free resolvent takes the form (in odd dimensions)
\begin{equation}\label{eq:rlamb}
(-\Delta-\lambda^2+i0)^{-1}(r)= \frac{e^{i\lambda r}}{r^{d-2}} \sum_{j=0}^{\frac{d-3}{2}} c_j\, (\lambda r)^j
\end{equation}
and the highest power $\lambda^{\frac{d-3}{2}}$ here corresponds to a $\frac{d-3}{2}$ derivative loss on~$V$. In the positive direction,
Erdo\smash{\u{g}}an and Green~\cite{EG} prove the dispersive bound in dimensions $d=5,7$ assuming that $V\in  C^{\frac{d-3}{2}}(\Rd^d)$ (zero energy
resonances cannot arise in dimensions $d\ge5$).

The case of low dimensions $d=1$ and $d=2$ always requires a separate analysis since the free resolvent in those cases
exhibits a zero energy singularity (more precisely, there is a zero energy resonance given by the constant state $f=1$).
We refer the reader to \cite{GS1},  Weder~\cite{Weder1},  d'Ancona, Fanelli~\cite{DaFa1} for the one-dimensional case, and~\cite{Sch1}
for dispersive estimates for the two-dimensional case provided zero energy is regular. Erdogan and Green~\cite{ErdGreen} carried out a more complete analysis 
of the dispersive decay in $\Rd^2$ allowing for $s$ and $p$-wave resonances at zero energy. This classification refers to nonzero solutions $\psi$ of $H\psi=0$
which  (i) are asymptotic to a nonzero constant at spatial $\infty$ for $s$-waves (ii) are in $L^q(\Rd^2)$ for all $q>2$ for $p$-waves. 
They showed that the $s$-wave resonance, which arises in the $V=0$ case, leads to the same $t^{-1}$ decay as in the free evolution, whereas the $p$-wave destroys this rate of decay. With Goldberg these authors also obtained such a classification in~$\Rd^4$. Finally, more recently Erdo\smash{\u{g}}an, Green, and Toprak have applied spectral methods to analyze the delicate dispersive decay of the Dirac operator, see~\cite{EGT}.

\subsubsection{The wave equation}
Starting with Beals and Strauss~\cite{BS}, \cite{Bea} many authors have considered the problem of proving the dispersive estimate~\eqref{eq:wave_disp} 
for equations $(\Box + V)u=0$, $(u,\partial_t u)(0)=(f,g)$ (it will suffice to set $f=0$). In~\cite{BS} and~\cite{Bea} the potential
is assumed to be either nonnegative or small (which excludes any spectral problems),
 as well as rapidly decaying and smooth. The result is of the form~\eqref{eq:wave_disp} but with slightly more derivatives on the data.
 Georgiev, Visciglia~\cite{GV} assume that $0\le V\le \la x\rangle^{-2-\epsilon}$ in three dimensions and obtain~\eqref{eq:wave_disp} for
  energies away from zero as well as Strichartz estimates for all energies.
Cuccagna~\cite{Cuc1} proves Strichartz estimates in three dimensions assuming that $|\partial^\alpha V(x)|\les \la x\rangle^{-3-\epsilon}$
for $|\alpha|\le 2$ and that zero energy is regular.
D'Ancona and Pierfelice~\cite{DAnPier} prove global dispersive~\eqref{eq:wave_disp} for $d=3$ assuming that $\|\min(V,0)\|_{K}<2\pi$
  but with $\dot B^1_{1,1}$ on the right-hand side.
Pierfelice~\cite{Pie} obtains the same result under the smallness assumption~\eqref{eq:kato_norm} (the arguments in~\cite{KS3} yield the same but with $\|\nabla g\|_1$ instead of the Besov norm).
D'Ancona, Fanelli~\cite{DaFa2} consider the wave and Dirac equations in three dimensions
\begin{align*}
u_{tt} - (\nabla + iA)^2 u + V u &=0\\
iU_t -{\mathcal D} U + M U &= 0
\end{align*}
respectively. Assuming smallness of $A,V,M$ but allowing nearly scaling-invariant singularities
of these functions both at zero and infinity (which are $|x|^{-1}$, $|x|^{-2}$, and $|x|^{-\frac12}$, respectively) the   $t^{-1}$
global decay is obtained but for data in weighted Soboloev and Besov spaces.
By the aforementioned results of Yajima et al.\ on the $W^{k,p}$-boundedness of the wave
operators one can obtain $L^p$ decay  estimates  for the wave equation from the free estimates~\eqref{eq:wave_disp}. 
Note that the Besov spaces are then defined relative to $H$ rather than the free Laplacian, but it is often possible to pass between the two. 
For a more recent reference on the integrated decay of waves, which also allows for magnetic perturbations, see d'Ancona's work~\cite{DAn}. 

\subsubsection{The case of singular zero energy}

Certain stability problems in physics lead to linear operators with a zero energy eigenvalue or resonance.
Examples are the energy critical wave equation $\Box u - u^5=0$ in $\Rd^{1+3}$ which admits the stationary solutions $W_\lambda(x):=\lambda (1+\lambda^2 |x|^2/3)^{-\frac12}$ for $\lambda>0$. Linearizing around
$W_\lambda$ leads to $H=-\Delta - 5W_\lambda^4$ which has $\partial_\lambda W_\lambda$ as a
resonant mode of zero energy. Another example is the critical Yang-Mills problem in dimensions $4+1$. 
It is therefore necessary to obtain dispersive bounds in this context as well. Note that the local decay of Section~\ref{subsec:local}
easily allows for this as the asymptotic expansions in time (as derived in~\cite{Mur}, \cite{JK} for example)  isolate the contributions of the threshold singularities and identifies them as being of finite rank. In case of $L^1\to L^\infty$ this required some additional work, see
\cite{ES1}, \cite{ES2}, and~\cite{Y4} for the case of the Schr\"odinger evolution. Yajima~\cite{Y4} obtains explicit  expressions for the  term $Bt^{-\frac12}$ which needs to be subtracted to obtain
the $t^{-\frac32}$ decay of the bulk ({\it explicit} here means that $B$ can be computed from the zero energy and resonance states). The wave equation in three dimensions is analyzed  in~\cite{KS3}. We recall the main linear result from the latter reference.

\begin{prop}
\label{prop:infty_decay} Assume that $V$ is a real-valued potential
such that  $|V(x)|\less \la x\ra^{-\kappa}$ where $\kappa>3$ is fixed but arbitrary.
If zero energy is regular for $H$, then
\[ \Big\| \frac{\sin(t\sqrt{H})}{\sqrt{H}} P_c f \Big\|_\infty \less t^{-1} \|f\|_{W^{1,1}(\Rd^3)}
\]
for all $t>0$.
Now assume that zero is a resonance  but not an
eigenvalue of $H=-\Delta+V$. Let $\psi$ be the unique resonance
function normalized so that $\int V\psi(x)\,dx=1$. Then there exists
a constant $c_0\ne0$ such that
\begin{equation}
\label{eq:1inf} \Big\| \frac{\sin(t\sqrt{H})}{\sqrt{H}} P_c f - c_0
(\psi\otimes \psi) f\Big\|_\infty \less t^{-1} \|f\|_{W^{1,1}(\Rd^3)}
\end{equation}
for all $t>0$.
\end{prop}

Several results exist on the  boundedness of the wave operators on $L^p$ in case zero energy is singular. However, they are limited to 
a smaller range of $p$ (in $d=3$ one needs $\frac32<p<3$), and are less useful for nonlinear applications, at least in three dimensions. On the other hand, in $\Rd^2$, Erdogan, Goldberg and Green~\cite{EGG} showed that the wave operators remain bounded in the full range $1<p<\infty$ if $0$ energy exhibits  only an $s$-wave resonance or only a zero energy eigenvalue. 

For the Klein-Gordon equation on the line with a non-generic decaying potential (i.e., the associated Schr\"odinger operator 
exhibits a zero energy resonance), an analogue of Proposition~\ref{prop:infty_decay} was obtained in~\cite{LLSS},  albeit for local decay. This is part of a larger body of work aiming at
understanding kink stability.

\section{Metric perturbations}\label{sec:metric}

If one  replaces  $-\Delta$ by the elliptic operator $H:=-\sum_{j,k=1}^d\partial_j (a_{jk}(x) \partial_k)$ then one encounters a new
obstruction to proving decay estimates in addition to the zero energy resonance or eigenvalue of Section~\ref{sec:low}: the phenomenon of trapping,
which is a {\em large energy problem}.
Trapping refers to the possibility that the classical Hamiltonian
\[
h(x,\xi) :=   \frac12\sum_{j,k=1}^d  a_{jk}(x) \xi^k \xi^j
\]
exhibits closed trajectories. More precisely, assuming symmetry $a_{jk}=a_{kj}$ one has the Hamiltonian equations
\[
\dot x :=  \sum_{j=1}^d  a_{j\cdot}(x) \xi^j, \qquad \dot \xi =  \frac12\sum_{j,k=1}^d  \nabla_x a_{jk}(x) \xi^k \xi^j
\]
which might exhibit time-periodic trajectories. 
To understand the
crucial effect of the existence of closed geodesics, we consider the
method of proving decay estimates using energy estimates:
$$
\frac{d}{dt} \langle u,A(t)u\rangle= \langle u,i[H,A(t)]u\rangle+\langle u,\frac{\partial
A(t)}{\partial t}u \rangle 
$$
where $u=u(x,t)$ is the solution of the Schr\"odinger equation, with
Hamiltonian $H$. A similar identity can be applied for the wave
equation, see~\cite {BSof}. Next, suppose the expectation of $A(t)$ is
bounded from above, uniformly in $t$, by $\| u\|_2$ and moreover, that the commutator is positive, in
the sense that
$$
i[H,A(t)]+\frac{\partial A(t)}{\partial t} \ge \theta B^{\star}B
$$
for some $\theta>0$ and some operator $B$. Upon integration
over time, we obtain an integrated decay estimate for $B$:
$$
\int_0^{\infty} \|Bu\|^2\, dt \leq c \|u(0)\|_2^2 
$$
The operator family $A(t)$ is variably called a multiplier, or a
propagation observable, or an escape function, or conjugate operator. 

To illustrate this further,  let $h(x,\xi)$ be a classical Hamiltonian on $\Rd^{2d}$.
If $(x(t),\xi(t))$ is an orbit under the Hamiltonian flow of~$h$, then
\[
\frac{d}{dt} a(x(t),\xi(t)) = \{ h, a\}(x(t),\xi(t))
\]
where the right-hand side is the Poisson bracket.
For the Euclidean case, i.e., $h(x,\xi)=\frac12\xi^2 + V(x)$ one can take $a(x,\xi)=x\cdot\xi=\{h,\frac12 |x|^2\}$ which gives $\{ h,a\}=2h - 2V - x\cdot\nabla V$.
Now suppose that $- 2V(x) - x\cdot\nabla V(x)\ge0$ for $|x|\ge R>0$, say.
Since $h$ is conserved, we conclude that a trajectory with $h=\alpha>0$ which remains in $|x|\ge R$ satisfies
\[
\frac{d^2}{dt^2} \frac12 |x(t)|^2 \ge 2\alpha
\]
and therefore $|x(t)|$ grows  linearly in~$t$.  This indicates that $(x(t),\dot x(t))$ undergoes scattering like a free particle. Under a short-range condition 
on $V(x)$, i.e., $|V(x)|\le C\langle x\rangle^{-1-\epsilon}$ this is indeed the case; i.e., all trajectories which are not trapped are asymptotically free.  See 
the book by Derezinski and Gerard~\cite{DerGer} for a systematic development of these
techniques in both classical and quantum mechanics.

Positive commutator methods  are
also used to prove refined average decay estimates which hold on subsets of the
phase space. Such estimates for the wave and Schr\"odinger equation
were first derived by Morawetz, using the radial derivative operator
 and the generator of the conformal group as multipliers.
 These multipliers also work if repulsive interactions are added.
However,  modifications are needed if trapped geodesics are present,
and usually only lead to weaker estimates. A major step in this
direction is the use of a {\em sharp localization of the energy}, due to
Mourre~\cite{Mourre}. The energy estimate can be obtained by taking
the derivative with respect to time of the expectation value of some
operator, also called propagation observable as in Sigal, Soffer~\cite{SSof1}. 
The remarkable paper by Hunziker, Sigal, and Soffer~\cite{HSS} presents a time-dependent approach
to Mourre theory based on the {\em commutator expansion lemma} of Sigal and Soffer. The latter refers to expressing 
$[f(A),B]$ through a   series of Taylor type involving higher-order commutators between $A$ and~$B$. 

A parallel
development to this approach was based on $\Psi$DO methods. In this approach one constructs a
function on the phase space which has positive Poisson bracket
 with the principal symbol of the Hamiltonian.
Then, one uses the quantized symbol of this function as a
propagation observable, and by means of $\Psi$DO theory, and in particular,
Garding's inequality,  passes to the desired smoothing (or limiting absorption) bound.
 Some of the earliest implementations of this approach are \cite{Colin},
\cite{Doi} and since then a vast literature has developed in this direction. 

The importance of a 
nontrapping condition is readily understood: it allows for the construction of monotonic
propagation observables, globally in the phase space. In the
presence of closed trajectories this is not possible. However, when
the trajectories are closed but (strongly) unstable, there is now substantial evidence that the decay estimates continue 
to hold in some sense. 

On the level of the resolvents, one considers $(H-z)^{-1}$ with $H$ a variable coefficient operator as above, with $a_{ij}$ a short-range perturbation of
 a constant elliptic symbol. Furthermore, assume that all classical Hamiltonian orbits of large initial velocity
are not trapped.  Then the limiting absorption
principle
\begin{equation}
 \sup_{\Im z>0,\Re z\ge N}\|\la \cdot\ran^{-\sigma} (H-z)^{-1} \chi_{I}(H) \la \cdot \ran^{-\sigma}\|_{2\to2} <\infty
\label{eq:lim_apMur}
\end{equation}
holds with $N>0$ and $\sigma>0$ sufficiently large, see Murata~\cite{Mur1}.   In fact, the nontrapping condition is necessary, see Theorem~2 in loc.~cit.,
and one also obtains~\eqref{eq:lim_apMur} for derivatives in~$z$ of the resolvent. The latter property then clearly implies local
decay on the time-evolution restricted to high energies.

In fact, while Doi~\cite{Doi}, Murata~\cite{Mur} show that smoothing
estimates and the usual decay estimates do not hold in the presence of
trapping, Ikawa~\cite{Ik} shows that one still obtains local decay
estimates for the Laplacian dynamics on $\Rd^n$ with several convex obstacles
removed. In the meantime, the microlocal analysis on manifolds with unstable closed geodesics,  of the resolvent of the Laplacian on the one hand,  and the
Schr\"odinger evolution on the other hand,  has grown into a vast area in and of itself which 
is intimately connected to the semiclassical analysis of scattering resonances.  See for example the
recent research monograph~\cite{Bony} on {\em Resonances for homoclinic trapped sets}, or Dyatlov's introduction to
the {\em fractal uncertainty principle}~\cite{Dyat}. 

In general relativity, unstable closed geodesics arise naturally in the study of the linear wave evolution on the background of both Schwarzschild and Kerr black holes. 
A substantial amount of work has accumulated around this topic, see for example the early works by Blue and Soffer, see~\cite{BSof} as well the very recent 
study of Price's law by Hintz~\cite{Hintz}. 
The latter paper was preceded by the work of  Tataru~\cite{Tat_GR}, as well as the results by Donninger, Soffer and the author~\cite{DSS2} on the spatially local, but temporally global,  decay of linear waves on Schwarzschild.  Metcalfe, Tataru, and Tohaneanu~\cite{MTT12} subsequently established Price’s law on nonstationary spacetimes with sufficient decay in a suitable sense.  Very recently, Angelopoulos,  Aretakis, Gajic~\cite{AAG} presented a ``physical space" approach to Price's law on Kerr space-times in contrast to Hintz's microlocal technique. We now set out to describe the author's results in more detail.

\subsection{Asymptotically conical surfaces of revolution}

A model case for the Schwarzschild manifold, Soffer, Staubach and the author~\cite{SSS1, SSS2} studied wave evolutions on surfaces of revolution with conic ends. 
Let $\Omega\subset \Rd^N$
be an embedded compact $d$-dimensional Riemannian manifold with metric $ds_\Omega^2$ and define the $(d+1)$-dimensional manifold
\begin{align*}
\calM &:=\{(x,r(x)\omega)\:|\: x\in\Rd,\; \omega\in\Omega\}, \\
ds^2 &=r^2(x)ds_{\Omega}^2 +(1+r'(x)^2)dx^2
\end{align*}
where $r\in C^\infty(\Rd)$ and $\inf_{x\in\Rd} r(x)>0$. We say that there is
a {\em conical end} at the right (or left) if
 \begin{equation}\label{eq:cone} r(x)=|x|\,(1+ h(x)),\quad
  h^{(k)}(x) =O(x^{-2-k}) \quad \forall\; k\geq 0
  \end{equation}
as $x\to\infty \; (x\to-\infty)$.

Of course one can consider cones with arbitrary opening angles but this adds
nothing of substance. Examples of such manifolds are given by surfaces of revolution with $\Omega=S^1$
such as the one-sheeted hyperboloid which satisfies $r(x)=\sqrt{1+x^2}$. They have the property that
the entire Hamiltonian flow on~$\calM$ is trapped on the
set $(x_0,r(x_0)\Omega)$ when $r'(x_0)=0$. From now we will only consider $S^1$ as cross-section $\Omega$  for
the sake of simplicity. The only difference from the general case is that instead of $\big\{ e^{\pm i\ell\theta},
\ell^2 \big\}_{\ell=0}^\infty$
 one has
a complete system  $\{Y_n,\mu_n\}_{n=0}^\infty $ of  $L^2$-normalized eigenfunctions and eigenvalues, respectively,  of~$\Delta_\Omega$.
In other words, $-\Delta_\Omega Y_n = \mu_n^2 Y_n$ where $0=\mu_0^2< \mu_1^2\le \mu_2^2\le\ldots$.

Note that we do not specify the local geometry of $\calM$, but only the asymptotic one at the ends. This allows
for very different behaviors of the geodesics. For the case of the one-sheeted hyperboloid, for example, the geodesic flow
around the unique periodic geodesic is hyperbolic in the sense of dynamical systems, whereas if we place a section of~$S^2$
in the middle of~$\calM$ then we encounter a set of positive measure in the cotangent bundle leading to stable periodic geodesics.
These two scenarios are depicted in Figure~2.
It is natural to ask to what extent this local geometry affects the dispersion of the flow.
The following result summarizes what is proved in~\cite{DSS1}, \cite{DSS2}
for the case of $\Omega=S^1$ (see those references for general compact~$\Omega$).

\begin{thm}\label{thm1}
Let $\calM$ be a surface which is asymptotically conical at both ends as defined above.
For each  $\ell\ge0$ and all $0\le\sigma\le \sqrt{2}\ell$,
there exist constants $C(\ell,\calM,\sigma)$ and $C_1(\ell,\calM,\sigma)$ such that for all $t>0$
\begin{align} \Vert w_\sigma\,
e^{it\Delta_{\calM}}\,  f\Vert_{L^{\infty}(\calM )} &\le \frac{C(\ell,\calM,\sigma)}{t^{1+\sigma}}\Big\Vert \frac{f}{w_\sigma }\Big\Vert_{L^1(\calM )} \label{eq:schr_d}\\
\Vert w_\sigma\, e^{it\sqrt{-\Delta_{\calM}} }\, f\Vert_{L^{\infty}(\calM )} &\le \frac{C_1(\ell,\calM,\sigma)}{t^{\frac{1}{2}+\sigma}}\Big(\Big\Vert \frac{\partial_x f}{w_\sigma }\Big\Vert_{L^1(\calM )} \nn \\
&\qquad + \Big\Vert \frac{f}{w_\sigma }
\Big\Vert_{L^1(\calM)}  \Big) \label{eq:wave_d} 
\end{align}
provided $f=f(x,\theta)=e^{i\ell\theta} \tilde f(x)$ where $\tilde f$ does not depend on~$\theta$. Here $w_\sigma(x):=\la x\ra^{-\sigma}$ are weights on~$\calM$.
\end{thm}

\begin{figure*}[ht]
\begin{center}
 \centerline{\hbox{\vbox{ \epsfxsize= 7.0 truecm \epsfysize=6.0
 truecm \epsfbox{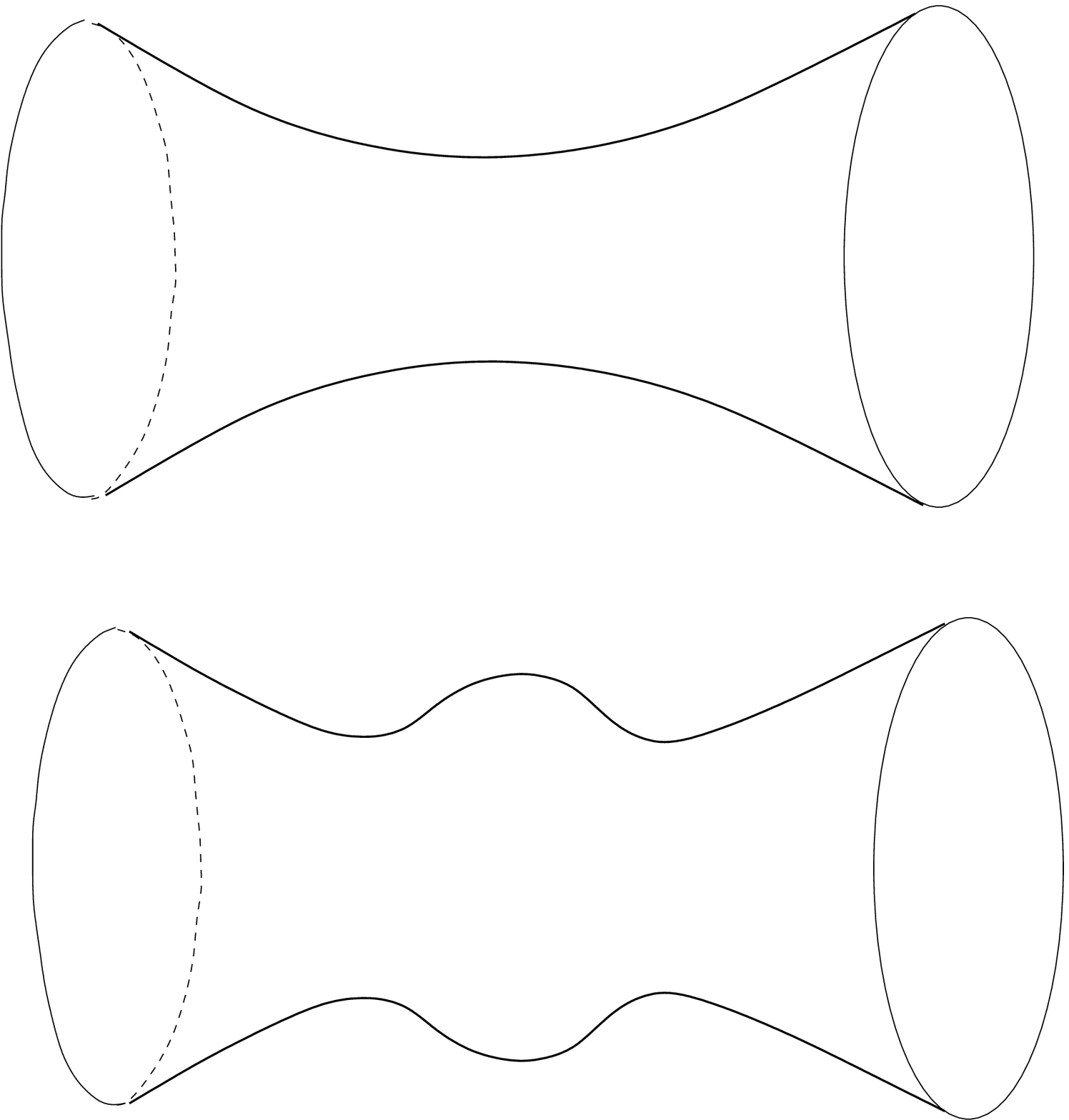}}}} \caption{Unstable versus stable geodesic flow}
\end{center}
\end{figure*}

In \eqref{eq:wave_d} one can obtain somewhat finer results by distinguishing between $\cos(t\sqrt{-\Delta_{\calM}})$ and $\frac{\sin(\sqrt{-\Delta_{\calM}})}{\sqrt{-\Delta_{\calM}}}$, see~\cite{DSS1} for statements of that kind.
 Needless to say, $\sigma=0$ is the analog
of the usual dispersive decay estimate for the Schr\"odinger and wave
evolutions on~$\Rd^2$. We remark that as in the case of the plane~$\Rd^2$, the free Laplacian $\Delta_{\calM}$ exhibits
a zero energy resonance which is, however, only visible at $\ell=0$ (this case is treated separately
in~\cite{SSS1}, whereas~\cite{SSS2}
studies $\ell>0$).

 Clearly, the {\em local} decay given by $\sigma>0$ has
no analog in the Euclidean setting and it also has no meaning for
$\ell=0$.
The restriction $\sqrt{2}{\ell}$ is optimal in Theorem~\ref{thm1}, at least for the Schr\"odinger equation,  and no faster decay can be obtained
than the one stated in~\eqref{eq:schr_d}. The $\sqrt{2}$-factor comes from the opening angle of $\frac{\pi}{4}$ and changing
that angle leads to different constants, namely $\frac{1}{\cos(\theta)}$ where $\theta$ is the opening angle of the asymptotic cone.

A heuristic explanation for the existence of this accelerated {\em local decay} is given by the geodesic flow combined with the
natural dispersion present in these equations. Indeed, the former will push any nontrapped geodesics into the ends (with $\ell$ playing the role
of the velocity of the geodesics), whereas the latter will spread any data which is initially highly localized around a periodic geodesic away
from it thus making it susceptible to the mechanism we just described.

What is not clear from this heuristic is whether or not the localized decay law should depend on the local geometry (by which we mean the
geometry which is not described by the asymptotic cones).
Theorem~\ref{thm1} shows
that this is not so, since the local decay is {\em fixed} and given by a specific power. Therefore, one sees that the
local geometry manifests itself exclusively through the constants $C(\ell,\calM,\sigma)$. 
This is natural, as one would expect a much longer waiting time before
the large~$t$ behavior of the theorem sets in if~$\calM$ exhibits stable geodesics.  In fact, the constant $C(\ell)$ grows exponentially
in that case as can be seen by solutions which are highly  localized (microlocally) around a periodic geodesic, see~\cite{Plamen} and~\cite{Sbie}.

In contrast, the methods of~\cite{DSS2} show that this constant grows like~$\ell^C$ if the manifold~$\calM$
has a unique periodic geodesic and is uniformly convex near it. This then allows one to sum up the estimates for each angular momentum
as described by the following theorem.

\begin{thm}\label{thm2}
Let $\calM$ be asymptotically conical at both ends as above and suppose that $\calM$ has a unique periodic geodesic and is uniformly convex near it. Then for all $t>0$,
and any $\eps >0$, and with $\calD:=1-\del_\theta^2$, 
\begin{align}
\Vert w_{1+\eps}
e^{it\Delta_{\calM}}\, w_{1+\eps}  f\Vert_{L^{2}(\calM )} &\le \frac{C(\calM,\eps )}{\la t\ra}\big\Vert{\calD\, f}\big\Vert_{L^2(\calM )} \label{eq:schr_d2}\\
\Vert w_{1}
e^{it\Delta_{\calM}}\, w_{1}  f\Vert_{L^{\infty}(\calM )} &\le \frac{C(\calM,\eps)}{ t}\big\Vert{\calD^{2+\eps}\, f}\big\Vert_{L^1(\calM )} \label{eq:schr_d3}
\end{align}
provided $f=f(x,\theta)$ is Schwartz on~$\calM$, say.
For the wave equation one has
\begin{align}
& \Vert w_{\frac12+\eps} e^{\pm it\sqrt{-\Delta_{\calM}}}\,w_{\frac12+\eps} f\Vert_{L^{2}(\calM )}  \nn \\
&\le \frac{C_1(\calM,\eps)}{\la t\ra^{\frac{1}{2}} }\Big(\big\Vert \calD^{\frac54}\, f'\big\Vert_{L^2(\calM )} + \big\Vert \calD^{\frac54}\, f
\big\Vert_{L^2(\calM)}  \Big) \label{eq:wave_d2}  \\
& \Vert w_{\frac12+\eps} e^{\pm it\sqrt{-\Delta_{\calM}}}\,w_{\frac12+\eps} f\Vert_{L^{\infty}(\calM )} \nn \\
&\le \frac{C_1(\calM,\eps)}{t^{\frac{1}{2}}}\Big(\big\Vert \calD^{\frac94+\eps}\,\partial_x f\big\Vert_{L^1(\calM )} + \big\Vert\calD^{\frac94+\eps}\, f
\big\Vert_{L^1(\calM)}  \Big) \label{eq:wave_d3}
\end{align}

\end{thm}

The weights $w_1$ and $w_{\frac12+\epsilon}$ appearing in~\eqref{eq:schr_d3} and~\eqref{eq:wave_d3}, respectively, are a by-product of our proof and can most
likely be removed. The origin of the weights in our method will be explained in Section~\ref{sec:suminell} below.
\begin{figure*}[ht]
\begin{center}
\centerline{\hbox{\vbox{ \epsfxsize= 6.0 truecm \epsfysize=5.0
truecm \epsfbox{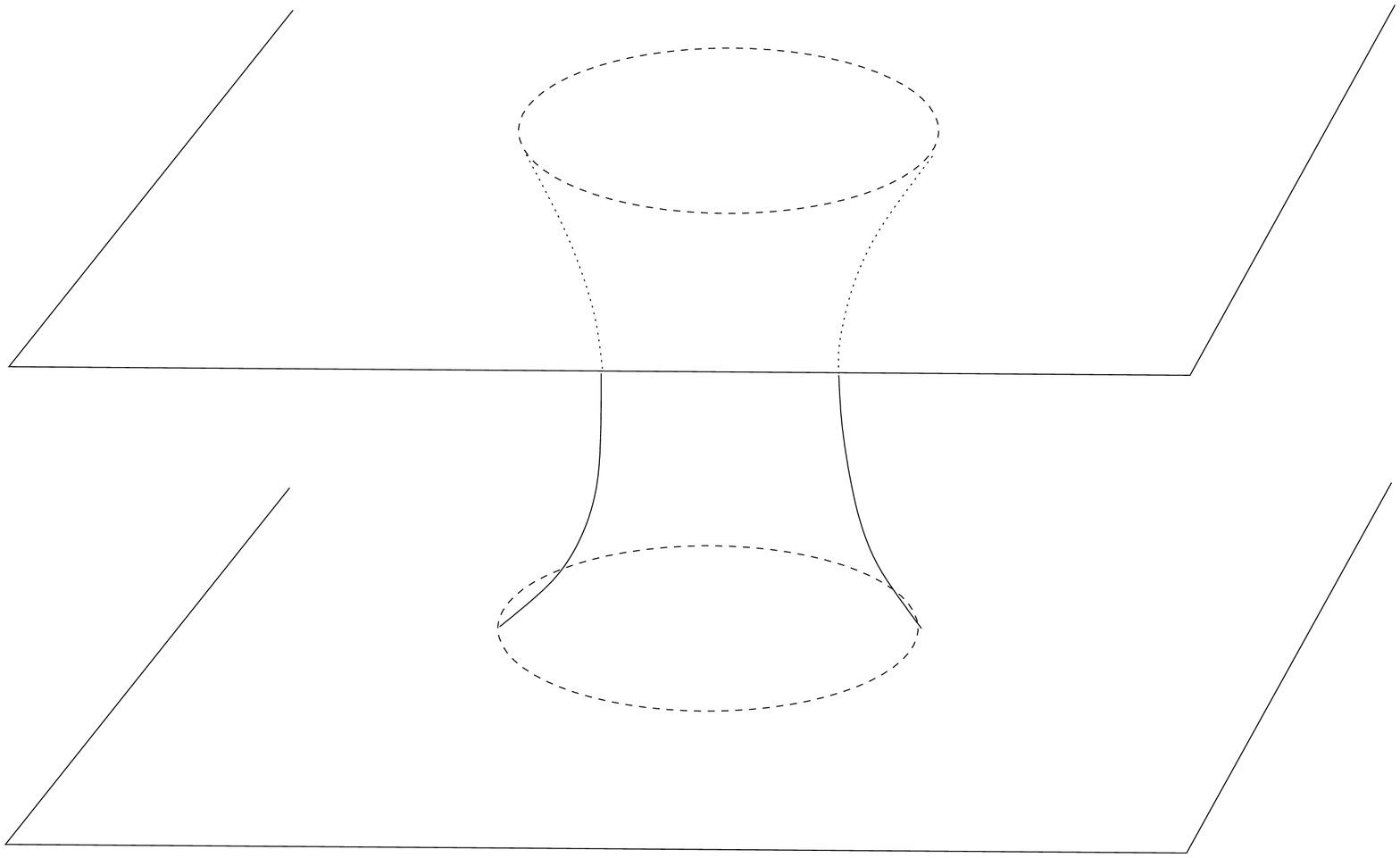}}}} \caption{Two planes joined by a neck}
\end{center}
\end{figure*}
One also obtains the accelerated decay rates which are better by $t^{-\sigma}$ as in Theorem~\ref{thm1} provided one puts in the weights as before, makes the number of derivatives required on the right-hand side depend on~$\sigma$, and provided the data are perpendicular to $e^{i\ell\theta}$ for $\sigma>\sqrt{2}\ell\ge0$.
We remark that one can think of the surfaces in Theorem~\ref{thm1} as two planes joined by a neck, see Figure~3. On the other hand,
the methods which are currently used to prove Theorems~\ref{thm1} and~\ref{thm2} do not extend to the case of more necks, as then
there is no clear way of separating variables.

There is no reason to expect that the number of derivatives required on the data in Theorem~\ref{thm2} is optimal, in fact it most certainly is not.
Heuristically speaking, these derivatives measure the spreading or {\em non-concentration} of solutions near hyperbolic orbits in dependence on the angular momentum, which is a quantum effect.  See for example Christianson's work~\cite{Chr} on this topic.

Doi~\cite{Doi} proved that the presence of trapping destroys the so-called local smoothing estimate for the Schr\"odinger evolution.
More precisely, he showed that one loses (even locally in time) the $\frac12$-derivative gain present in~$e^{it\Delta}$. Note that this does not constitute a contradiction to Theorem~\ref{thm2} as the latter does not claim any gain of regularity (on the contrary, we lose angular derivatives). In a similar vein, Burq, Guillarmou, Hassell~\cite{BGH}
proved
that Strichartz estimates may remain valid on metrics with trapping.

We now describe the method of proof leading to Theorem~\ref{thm1}. Later we will discuss how to obtain Theorem~\ref{thm2}, which
requires considerably more work. We will then also describe the result~\cite{DSS2} for linear waves on Schwarzschild, which is very close
to Theorem~\ref{thm2}.

To begin with, let $\xi$ be arclength along a generator of~$\calM$. Then the Laplacian takes the form
\[
\Delta_{\calM} = \frac{1}{r(\xi)} \partial_\xi (r(\xi)\partial_\xi) + \frac{1}{r^2(\xi)} \Delta_{S^2}
\]
Now
\[
e^{-i\ell\theta} r^{\frac12}(\xi) \Delta_{\calM} (r^{-\frac12}(\xi) e^{i\ell\theta} f(\xi)) =  \calH_{\ell} f
\]
with
\begin{equation}\label{eq:calHell}
\calH_\ell = -\partial_\xi^2 + V_{\ell},\qquad V_{\ell}(\xi) = \frac{2\ell^2-\frac14}{\la \xi\rangle^2} + O(\la \xi\rangle^{-3})
\end{equation}
where each $\xi$-derivative of the $O(\cdot)$-term gives one extra power of~$\xi$ as decay.
\begin{figure*}[ht]
\begin{center}
\centerline{\hbox{\vbox{ \epsfxsize= 6.5 truecm \epsfysize=5.8
truecm \epsfbox{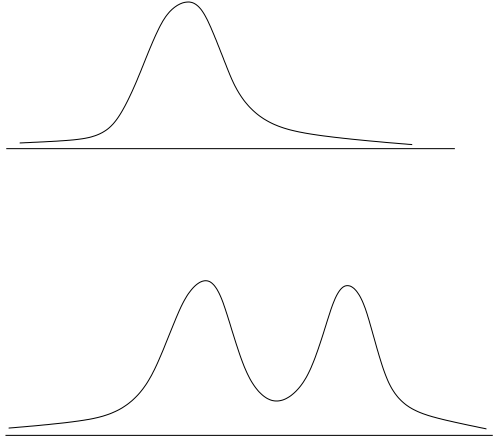}}}} \caption{Potentials corresponding to the surfaces in Figure~2}
\end{center}
\end{figure*}
 We remark that the leading $\la \xi\rangle^{-2}$ decay
is {\em critical} for several reasons. For us most relevant is the behavior of the Jost solutions as the energy~$\lambda^2$ tends to zero; in fact
these Jost solutions are continuous in $\lambda$ around~$\lambda=0$ provided the decay of the potential is at least $\la \xi\rangle^{-2-\epsilon}$ for some $\epsilon>0$. At $\epsilon=0$ this property is lost -- which is precisely what allows for the accelerated decay of Theorem~\ref{thm1}. To be more specific, one first reduces Theorem~\ref{thm1} (at least the Schr\"odinger bound~\eqref{eq:schr_d}, the wave equation
being similar) via the spectral theorem to the point-wise bound
\begin{align}
\nn \sup_{\infty>\xi\ge\xi'>-\infty} &(\la \xi\ra \la \xi'\ra )^{-\frac12}
\Big| \int_0^\infty e^{it\lambda^2} \Im \Big[ \frac{f_{+,\ell}(\xi,\lambda) f_{-,\ell}(\xi',\lambda)}{ W_{\ell}(\lambda)} \Big] \\
&\qquad \, \lambda\, d\lambda \Big| \le C_\ell \, t^{-1-\sigma} \label{eq:xixi'} 
\end{align}
where $C_\ell$ is a uniform constant and $\sigma=\sqrt{2}\ell$.
Here $f_{\pm,\ell}$ are the (outgoing) Jost solutions, which satisfy $\calH_{\ell} f_{\pm,\ell} =\lambda^2 f_{\pm ,\ell}$ and $f_{\pm,\ell}\sim e^{\pm i\lambda \xi}$ as $\xi\to\pm\infty$. Moreover, $W_\ell(\lambda)$ is the Wronskian of $f_+,f_-$.
We remark that the quantity inside the absolute values in~\eqref{eq:xixi'} is exactly
\[
\int_0^\infty e^{it\lambda^2} E(d\lambda^2)(\xi,\xi')
\]
  where $E(d\lambda^2)(\cdot,\cdot)$ is the kernel of the spectral resolution of $\calH_\ell$.
  As usual,
  \[
  f_+(\xi,\lambda) = e^{i\xi\lambda} + \int_{\xi}^\infty \frac{\sin(\lambda(\xi'-\xi))}{\lambda} V(\xi') f_+(\xi',\lambda)\, d\xi'
  \]
  From this formula, one immediately sees the aforementioned discontinuity at $\lambda=0$ since $\xi V(\xi)\not\in L^1(0,\infty)$.
  Setting $\xi=\xi'=0$,
  \eqref{eq:schr_d} of Theorem~\ref{thm1} reduces to the
 standard stationary phase type bound (with $\nu:=\sqrt{2}\,\ell$)
 \[
 \Big|\int_0^\infty e^{it\lambda^2} \lambda^{1+2\nu}
 \chi(\lambda)\,d\lambda \Big|\le C t^{-1-\nu}
 \]
 where $\chi$ is a smooth cut-off function to the interval $[0,1]$,
 say. To see why the spectral measure should be as flat as $\lambda^{1+2\nu}\,d\lambda$, let us
 first give an informal proof of the fact that
 \begin{equation}\label{eq:Welllambda}
 W_\ell(\lambda)=c\lambda^{1-2\nu}(1+o(1))\qquad \lambda\to0
 \end{equation}
where $c\ne0$.  Since this Wronskian appears in the denominator of the resolvent, it at least serves as
an indication that the spectral measure might be this small for small~$\lambda$ (one has to be very careful here,
since the numerator is of the same size -- however, the {\em imaginary part} of the resolvent has the desired size $O(\lambda^{2\nu})$).
To begin with, recall from basic scattering theory  that the Wronskian is given by
\begin{equation}
\label{eq:Wlam}
W(\lambda)=\frac{-2i\lambda}{T(\lambda)}
\end{equation}
where $T(\lambda)$ is the transmission coefficient, see Figure~5 (in that figure the dashed line is supposed to indicate an energy level~$k^2$, and the
turning points are defined as the projections of the intersection of the graph with that line). By the so-called WKB approximation, one has
to leading order that $T(\lambda)=e^{-S(\lambda)}$ with the action $S$ given
by
\begin{align*}
S(\lambda) &= \int_{x_0}^{x_1} \sqrt{\nu^2 \la y\rangle^{-2}-\lambda^2}\, dy \\
&= \hbar^{-1} \int_{x_0}^{x_1} \sqrt{2 \la y\rangle^{-2}-\hbar^2\lambda^2}\, dy
\end{align*}
with $x_0<0<x_1$ being the turning points which are defined as  $V(x_0)=V(x_1)=\lambda^2$. Note that we modified the potential by
removing the cubic corrections as well as the $-\frac14 \la \xi\ran^{-2}$ part of the potential (the latter obviously requiring
some justification).  Furthermore, we used that $\nu=\sqrt{2}\ell$ and assumed $\ell>0$. As a result,
\[
S(\lambda) = 2\nu |\log\lambda|(1+o(1))\qquad \lambda\to0
\]
which then gives \eqref{eq:Welllambda} to leading order.
To justify the removal of the $\frac14 \la \xi\ran^{-2}$-part of the potential~$V_\ell$, we simply note that the usual WKB ansatz for
the zero energy solutions of~$\calH_\ell$, viz.~$\calH_\ell f=0$ is the approximate equality
\[
f(\xi) \simeq V_\ell^{-\frac14}(\xi) e^{\pm\int_1^\xi \sqrt{V_\ell(\eta)}\, d\eta}
\]
In view of \eqref{eq:calHell} one obtains the asymptotic behavior $\xi^{\frac12 \pm \sqrt{\nu^2-\frac14}}$ as $\xi\to\infty$.
On the other hand, the exact solutions of
\[
-f''(\xi) +  \frac{\nu^2-\frac14}{ \xi ^2} f(\xi)=0
\]
are of the form $\xi^{\frac12 \pm \nu}$. The WKB approximation can therefore only be correct provided the $-\frac14 \xi^{-2}$ term is removed from the
potential~$V_\ell$ (for a precise rendition -- with control of error terms -- of this heuristics discussion see Section~2 of~\cite{SSS2}).
Another important comment concerning $V_\ell$ is that \eqref{eq:Wlam}, while true to leading order {\em semi-classically} as $\hbar=\ell^{-1}\to0$, {\em provided}
the energy $\lambda>\lambda_0>0$ (where the latter is fixed), does not necessarily hold as $\lambda\to0$.
The key property here is that $\calH_\ell$ does not have {\em a zero energy resonance} which means that there is no globally subordinate (or 
recessive) solution. This refers to solutions of the slowest
allowed growth at both ends. For example, consider the operator $H=-\frac{d^2}{dx^2}+V$ where $V=\bar{V}$ satisfies $(1+|x|)V(x)\in L^1(\Rd)$. Then by the usual Jost/Volterra perturbation analysis, there is a fundamental system of solutions to $Hf=0$ consisting of $f_1(x)\sim x$, respectively $f_2(x)\sim 1$ as $x\to\infty$. Thus $f_2$ is the unique (up to nonzero factors) subordinate solution at $+\infty$. A resonance at $0$ energy therefore occurs if $Hf=0$ admits a solution $f\ne0$ which is asymptotic to a constant for both $x\to\pm\infty$. Since the only other option would be some linear growth at either end, this is equivalent to $f\in L^\infty(\Rd)$. This is a {\em universal} characterization of $0$ energy resonances through solutions of $Hf=0$ even if $V$ violates $(1+|x|)V(x)\in L^1(\Rd)$ as in Bessel-type potentials arising in most problems discussed in this note, or for that matter, for $V$ which are strongly singular. The latter  means  that $V$ is locally bounded (for simplicity) but $\int_{-\infty}^\infty |V(x)|\, dx=\infty$,  and we assume that $H$ is limit point at both ends. Depending on the specific choice of $V$, one needs to find a fundamental system of $Hf=0$ at both $\pm\infty$ and then select the subordinate solution. A resonance is characterized by a solution which is globally subordinate. 
\begin{figure*}[ht]
\begin{center}
 \centerline{\hbox{\vbox{ \epsfxsize= 7.8 truecm \epsfysize=4.8
 truecm \epsfbox{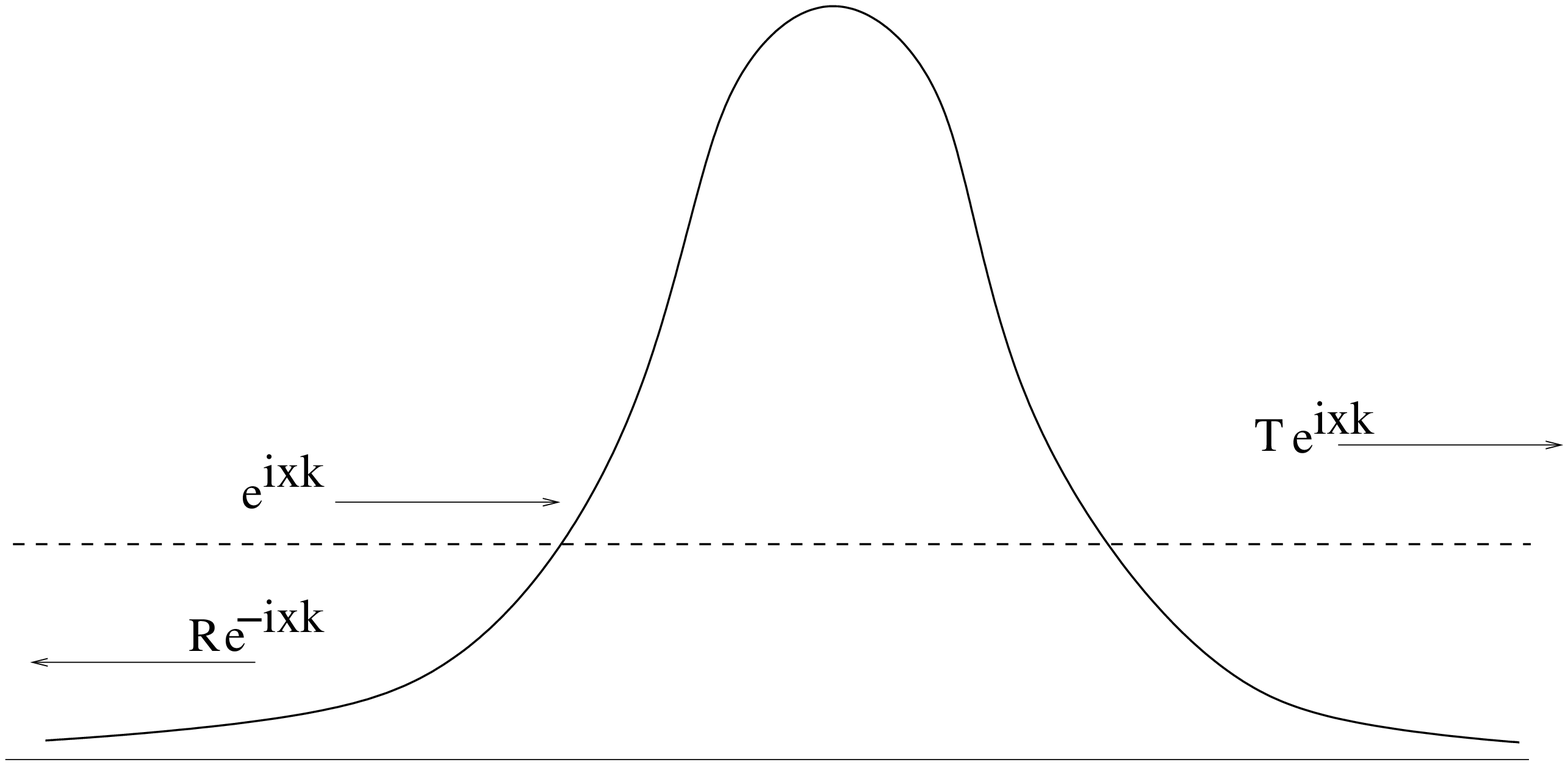}}}} \caption{Reflection and transmission coefficients}
\end{center}
\end{figure*}

While our discussion has been largely heuristic, we emphasize that~\eqref{eq:Welllambda} is proved in~\cite{SSS2} by means of an
asymptotic description of the Jost solutions as $\lambda\to0$. Moreover, it is shown there that the constant~$c$ in~\eqref{eq:Welllambda}
vanishes in case of a zero energy resonance which indicates that the WKB approximation fails in that case as $\lambda\to0$.
Finally, we emphasize that  the only natural small parameter in~\cite{SSS2} for  {\em fixed} $\ell\ge1$ is the energy $\lambda$. This is in contrast to the summation problem in~$\ell$ where $\hbar:=\ell^{-1}$ represents another (and most important) small
parameter. In fact, for large $\ell$ the errors in the WKB approximations are controlled in terms of this small parameter rather
than in terms of the small energy (we will return to this matter below). In order to be able to distinguish the two potentials in Fig.~4, or manifolds with distinct local geometries in Fig.~2 we therefore need to obtain precise asymptotics for the Jost solutions and the spectral measure for {\em both} small energies $0<\lambda<\lambda_0$ and all large $|\ell |$, {\em simultaneously}. This sets these problems apart from most of the semi-classical literature in several ways: (i) it is not enough to compute the limit $\hbar\to 0$. In fact,  we need a precise asymptotic representation of the Jost solutions uniformly in small $\hbar$ and all energies. This will be explained in more detail below in the summation section. (ii) The need for uniform control for all small energies is also in stark contrast to the literature which typically restricts any semi-classical analysis to positive or large energies. 

The rigorous proof of~\eqref{eq:Welllambda}
 proceeds by means of a classical matching method. To be more specific, consider the
Schr\"odinger operator on the line (for notational convenience we write~$x$ instead of~$\xi$)
\begin{align*}
\calH_\nu &= -\partial_x^2 + (\nu^2-\frac14)\la x\ran^{-2}-U_\nu(x), \\
\frac{d^k U_\nu(x)}{d x^k}
& =O(x^{-3-k})
\end{align*}
for all $k\ge0$
as $x\to\pm\infty$ and with $\nu>0$ fixed. To decribe the Jost solution $f_{+,\nu}(x)$ on the interval $x\ge1$ we start from
the zero energy solutions
\begin{align*}
u_{0,\nu}^{+}(x)  &= x^{\frac12+\nu}(1+O(x^{-\alpha})), \\
  u_{1,\nu}^{+}(x) &= x^{\frac12-\nu}(1+O(x^{-1}))  \text{\ \ as\ }x\to\infty
\end{align*}
which form a fundamental system of $\calH_\nu f=0$ (and with $\alpha:=\min(1,2\nu)$).
Next, one perturbs these solutions with respect to the energy~$\lambda$. More specifically,
one shows via Volterra iteration that there is a basis $\{u_{0,\nu}^{+}(x,\lambda) , u_{1,\nu}^{+}(x,\lambda) \}$
of solutions to the equation $\calH_\nu f =\lambda^2 f$ which satisfy (at least for $\nu>1$)
\begin{equation}\label{eq:ujnu}
u_{j,\nu}^{+}(x,\lambda)= u_{j,\nu}^{+}(x) (1+ O(\lambda^2 x^2))
\end{equation}
on the interval $1\le x\ll \lambda^{-1}$ (we are only considering small $\lambda$ for now).
Clearly, one has
\[
f_{+,\nu}( x ,\lambda ) =a_{+,\nu}(\lambda )u_{0,\nu}^+( x ,\lambda )+b_{+,\nu}(\lambda )u_{1,\nu}^+( x ,\lambda )
\]
where the coefficients are given by
\begin{equation}
\begin{aligned}
\label{eq:Wfnu}
a_{\pm,\nu}(\lambda ) &= -W(f_{\pm,\nu}(\cdot
,\lambda ),u_{1,\nu}^{\pm}(\cdot ,\lambda )),\\
b_{\pm,\nu}(\lambda
)&= W(f_{\pm,\nu}(\cdot ,\lambda ),u_{0,\nu}^\pm(\cdot ,\lambda
))
\end{aligned}
\end{equation}
The aforementioned {\it matching} means nothing else than computing these Wronskians.
The point where they are computed is chosen to be $\lambda^{-1+\epsilon}$ with~$\epsilon>0$ small and fixed. On the one
hand,  this choice guarantees that the errors in~\eqref{eq:ujnu} are $O(\lambda^{2\epsilon})$, which is admissible. On the
other hand, it requires that we obtain a sufficiently accurate description of the Jost solutions on $[\lambda^{-1+\epsilon},\infty)$.
The latter is accomplished by comparing the outgoing Jost solution of the operator $\calH_\nu$ to that of $\calH_{0,\nu}$ given by
\[
\calH_{0,\nu} := -\partial_x^2 + (\nu^2-\frac14) x^{-2}
\]
The outgoing Jost solution of this operator on~$\xi\ge1$ equals
\[
\sqrt{\frac{\pi}{2}}\, e^{i(2\nu+1)\pi
/4}\sqrt{\xi\lambda}\,H^{(+)}_\nu(\xi\lambda)
\]
which is asymptotic to $e^{i \xi\lambda}$ as $\xi\to\infty$. Here $H^{(+)}_\nu(z)=J_\nu(z)+iY_\nu(z)$ is the usual
Hankel function. Carrying out the perturbative analysis with $\calH_{0,\nu}$ as giving the leading order allows one
to approximate $f_+(\xi,\lambda)$ with small errors on the interval $(\lambda^{-1+\epsilon},\infty)$.
With this asymptotic representation in hand, one now has the following result, see Proposition~3.12 in~\cite{SSS2}.

\begin{prop}
  \label{prop:wronski} Let $\beta_\nu := \sqrt{\frac{\pi}{2}}\, e^{i(2\nu+1)\pi
  /4}$. With nonzero real constants
  $\alpha_{0,\nu}^+,\beta_{0,\nu}^+$, and some
  sufficiently small $\eps>0$,
  \begin{equation}
    \label{eq:abW}
    \begin{aligned}
      a_{+,\nu}(\lambda) &= \lambda^{\frac12+\nu} \beta_\nu(\alpha_{0,\nu}^+ +
      O(\lambda^\eps) + iO(\lambda^{(1-2\nu)\eps})) \\
 b_{+,\nu}(\lambda) &= i\lambda^{\frac12-\nu} \beta_\nu(\beta_{0,\nu}^+ +
      O(\lambda^\eps) + iO(\lambda^{(1+2\nu)\eps}))
  \end{aligned}
  \end{equation} as $\lambda\to0+$
  with real-valued $O(\cdot)$ which behave like symbols under
  differentiation in~$\lambda$. The asymptotics as $\lambda\to0-$
  follows from that as $\lambda\to0+$ via the relations
  $a_{+,\nu}(-\lambda)=\overline{a_{+,\nu}(\lambda)}$,
  $b_{+,\nu}(-\lambda)=\overline{b_{+,\nu}(\lambda)}$.
\end{prop}

Analogous expressions hold for $a_{-,\nu}$
and $b_{-,\nu}$ which of course refers to the solutions on $x\le -1$.
From these expansions, one then concludes the following statement for the Wronskian
between $f_+(\cdot,\lambda)$ and $f_-(\cdot,\lambda)$:
\[
W_\nu(\lambda)=i e^{i\nu\pi}\, \lambda^{1-2\nu} (W_{0,\nu} + O_{\C}(\lambda^\eps)) \text{\ \ as\ \ }\lambda\to0+
\]
Here $W_{0,\nu}$  is a real constant and $O_\C(\lambda^\eps)$ is complex valued and  of symbol type (meaning
that each derivative loses one power). Most importantly, $W_{0,\nu}=0$ if and only if zero is a resonance of $\calH_\nu$.
For the case of surfaces of revolutions, it is easy to exclude zero energy resonances of the associated Schr\"odinger
operator, at least for $\ell\ge1$. In fact, with $\calH_\ell$ denoting the operator obtained for fixed angular momentum $\ell\ge1$,
\[
\calH_{\ell} (r^{\frac12} e^{\pm \ell y})=0,\quad y(\xi)=\int_0^\xi \frac{d\eta}{r(\eta)}
\]
Because $y$ is odd, the smaller branch at $\xi=\infty$
has to be the larger one at $\xi=-\infty$ which places us in the nonresonant case.
It is perhaps worth mentioning that the potentials arising from surfaces of revolution do not need to be nonnegative (for positive
potentials it is evident that zero is not a resonance).
In fact, if $\calM$ has very large  curvature  then the potential can be negative.
We remark that for $\ell=0$ it is proved in~\cite{SSS1} that
\[
W_0(\lambda )=2\lambda \left( 1+ic_3+i\frac{2}{\pi}\log\lambda
\right)+O(\lambda^{\frac{3}{2}-\eps}) \text{\ \ as\ \ }\lambda\to0+
\]
On a technical level, the logarithmic term in $\lambda$ makes the
$\ell=0$ case somewhat harder to analyze than the cases $\ell\geq1$. Not surprisingly, in  proving dispersive estimates
for $-\Delta_{\Rd^2}+V$ one encounters similar logarithmic issues,
see~\cite{Sch1}.

In conclusion, we would like to stress that the estimates in~\cite{SSS2} produce constants that grow very rapidly in~$\ell$, somewhat
faster than $e^{\ell^2}$, to be precise. This is due to a number of sources. First, for the small energy analysis we just described
to work, one needs to chose the energy cut-off $\lambda_0=\lambda_0(\ell)$ to depend on~$\ell$ which already introduces large constants into the proof.
 Second, for energies $\lambda>\lambda_0(\ell)>0$ one uses a very crude method,
namely termwise estimation of a Born series which cannot distinguish the sign of the potential. Even replacing
the crude Born series by something more elaborate would not make much of a difference. Indeed, by the preceding discussion
the two manifolds in Figure~2 behave very differently as far as the dependence of the constant on~$\ell$ is concerned.

Since the {\em small energy matching} method outlined above cannot easily distinguish between these manifolds, we shall now discuss an approach
that is capable of differentiating between them, albeit only for large~$\ell$. For this reason, the finite $\ell$ analysis of~\cite{SSS1} and~\cite{SSS2} is needed in the proof of Theorem~\ref{thm2}.

\subsection{Summation over all angular momenta}\label{sec:suminell}

We shall now prove Theorem~\ref{thm2}.
We will follow~\cite{DSS2} and sketch how to obtain~\eqref{eq:schr_d2} and~\eqref{eq:schr_d3}, with the case of the wave equation being similar.
With $V_\ell$ as in~\eqref{eq:calHell}, we  claim the following bound:
\begin{equation}
\label{eq:claimL2}
\int_{-\infty}^\infty \la \xi\ran^{-2} \big| e^{it\calH_\ell} u(\xi)\big|^2\, d\xi \les \la t\ra^{-2} \ell^{4} \int_{-\infty}^\infty \la \xi\ran^{2} |u(\xi)|^2\, d\xi
\end{equation}
The proof of~\eqref{eq:claimL2} will be discussed below. Taking it for granted, suppose that $f$ is a Schwartz function on~$\calM$ and write
\[
f(\xi,\theta) = \sum_{\ell=-\infty}^\infty e^{i\ell\theta} f_\ell(\xi) = \sum_{\ell=-\infty}^\infty e^{i\ell\theta} r^{-\frac12}(\xi)  u_\ell(\xi)
\]
Then
\begin{align*}
e^{it\Delta_{\calM}} f &= \sum_{\ell=-\infty}^\infty e^{it\Delta_{\calM}} \big[e^{i\ell\theta} r^{-\frac12}(\xi) u_{\ell}(\xi)\big ] \\
&= \sum_{\ell=-\infty}^\infty   e^{i\ell\theta} r^{-\frac12}(\xi) \big[ e^{it\calH_\ell}  u_{\ell}\big](\xi)
\end{align*}
whence
\begin{align*}
& \big\| w_1 e^{it\Delta_{\calM}} f\big\|^2_{L^2(\calM)} \\
&= \frac{1}{2\pi }\!\int_{-\infty}^\infty \!\int_0^{2\pi} \!\!w_1^2(\xi) \Big| \!\sum_{\ell=-\infty}^\infty  \!\!\! e^{i\ell\theta} r^{-\frac12}(\xi)\! \big[ \!e^{it\calH_\ell}  u_{\ell}\big](\xi)  \Big|^2 \!\!r(\xi)\,d\xi d\theta \\
&\les \sum_{\ell=-\infty}^\infty \int_{-\infty}^\infty \la \xi \rangle^{-2} \Big| e^{it\calH_\ell} u_\ell (\xi)\Big|^2\, d\xi \\
&\les \sum_{\ell=-\infty}^\infty \la t\rangle^{-2} \la \ell\rangle^4 \int_{-\infty}^\infty \la \xi\ran^{2}  |f_\ell(\xi)|^2 r(\xi)\, d\xi \\
&\les \la t\rangle^{-2} \!\!\sum_{\ell=-\infty}^\infty \! \int_{-\infty}^\infty  \!\la \ell\rangle^4 w_{-1} (\xi)^2 \Big|\! \int_0^{2\pi}\! f(\xi,\theta) e^{-i\ell\theta}\,d\theta\Big|^2 r(\xi)\,d\xi \\
&\les \la t\rangle^{-2} \| w_{-1} (1-\del_\theta^2)f\|_{L^2(\calM)}^2
\end{align*}
which is~\eqref{eq:schr_d2}.
To prove \eqref{eq:claimL2},  it is clear from Theorem~\ref{thm1} that it suffices to consider $\ell$ large, say $|\ell|\ge \ell_0\gg1$.
Fixing such an~$\ell$,  one switches to a semi-classical representation, via the identity
\[
e^{it\calH_\ell} = e^{i\frac{t}{\hbar^2} \calH(\hbar)}, \qquad \calH(\hbar):=-\hbar^2\del_\xi^2+\hbar^2 V_\ell
\]
where $V_\ell$ is as in~\eqref{eq:calHell} and with $\hbar:=\ell^{-1}$. By construction, $V(\xi,\hbar):=\hbar^2 V_\ell(\xi)$ has the property
that its maximal height is now essentially fixed at $V_{\rm max}(\hbar)= V_{\rm max}(0) + O(\hbar^2)$ with $V_{\rm max}(0)\simeq 1$.
The essential property of the potential is that it has a unique nondegenerate maximum, i.e., it looks like the one on top in Figure~4.

For the remainder of this section, $\hbar$ will be small.
From the spectral representation one has
\begin{equation}
e^{i\frac{t}{\hbar^2} \calH(\hbar)} = \frac{2}{\pi} \hbar^{-2} \int_0^\infty e^{i\frac{t}{\hbar^2} E^2}
\Im \Big[
\frac{f_{+}(x,E;\hbar)f_{-}(x',E;\hbar)}{W(f_{+}(\cdot,E;\hbar),f_{-}(\cdot,E;\hbar))}
\Big] \, E\,dE    \label{eq:semievol}
\end{equation}
with $f_{\pm}$ being the outgoing Jost solutions for
the semi-classical operator $\calH(\hbar)$ which means that
\begin{align*}
 (-\hbar^2 \del_x^2 + V(x;\hbar))f_{\pm}(x,E;\hbar) &= E^2 f_{\pm}(x,E;\hbar) \\
 f_{\pm}(x,E;\hbar) &\sim e^{\pm \prefac Ex} \qquad x\to \pm \infty
\end{align*}
With $\epsilon>0$ fixed and small (independently of $\hbar$), one now considers energies $0<E<\epsilon$ (low), $\epsilon<E<100$ (intermediate), and $E>100$ (large) separately.
The middle interval is further split into energies $\epsilon<E<V_{\rm max}(0)-\epsilon$, $V_{\rm max}(0)-\epsilon<E<100$, respectively. The latter interval
is to some extent the most important of all as it contains the nondegenerate maximum of the potential $V(\hbar)$. We shall see that it is precisely this maximum which
determines the number of derivatives lost in the process of summing over~$\ell$.

The easiest region is $E>100$. Indeed, for these energies the potential is essentially negligible and a classical WKB approximation reduces matters
to the free case. This means (again heuristically) that~\eqref{eq:schr_d2} is a consequence of the $L^1\to L^\infty(\Rd^2)$ bound on $e^{it\Delta_{\Rd^2}}$ which
explains the weights $w_{1+\epsilon}$.

\subsubsection{WKB in the doubly asymptotic limit $\hbar\to0$ and $E\to0$}
The low-lying energies $0<E<\epsilon$ are also treated by means of WKB, but there one faces the difficulty that the WKB approximation of the generalized eigenfunctions
needs to be accurate in the entire range $0<E<\epsilon$ and $0<\hbar<\hbar_0$. There exists an extensive literature on the validity of the WKB approximation provided
the energy stays away from zero, i.e., $E>E_0>0$ uniformly in~$\hbar$, see for example~\cite{Olver} or Ramond~\cite{Ram}.
However, the issue of controlling all errors in the WKB method uniformly in small $\hbar$ {\em and small}~$E$ does not seem to have been considered before.
For the problem of sending $E\to0$ it is of
course most relevant that the potential has the (critical) inverse square decay, as was already apparent in the discussion of the matching method in the previous section.

This  lead
Costin, Schlag, Staubach, and Tanveer~\cite{CSST} to  carry out a systematic analysis of this two-parameter WKB problem for inverse square potentials. More
specifically, they considered the scattering matrix
\[
 \Sigma(E;\hbar) = \left [ \begin{matrix}  t(E;\hbar) &  r_-(E;\hbar) \\ r_+(E;\hbar) & t(E;\hbar) \end{matrix} \right ]
= \left [ \begin{matrix} \Sigma_{11}(E;\hbar) & \Sigma_{12}(E;\hbar) \\ \Sigma_{21}(E;\hbar) & \Sigma_{22}(E;\hbar) \end{matrix} \right ]
\]
for the semiclassical operator
\[
P(x,\hbar D) := -\hbar^2 \frac{d^2}{dx^2} + V(x)
\]
with inverse square $V$ (asymptotically, as $|x|\to\infty$) and obtained the following result.

\begin{thm}\label{thm:main}
Let $V\in C^\infty(\Rd)$ with $V>0$ and $V(x)=\mu_{\pm}^2 x^{-2} + O(x^{-3})$
as $x\to\pm\infty$ where $\mu_+\ne0$, $\mu_-\ne0$
and $\partial_x^k O(x^{-3}) = O(x^{-3-k})$ for all $k\ge0$.
Denote
\begin{equation}\label{eq:V0_def}
 V_0(x;\hbar):= V(x) + \frac{\hbar^2}{4}\la x\ra^{-2}
\end{equation}
and let  $E_0>0$ be such
that for all $0<E<E_0$ and $0<\hbar<1$, $V_0(x;\hbar)=E$ has a unique pair of solutions,
which we denote by $x_2(E;\hbar)<0<x_1(E;\hbar)$. Define
\begin{equation}\label{eq:STdef} \begin{split}
S(E;\hbar) &:= \int_{x_2(E;\hbar)}^{x_1(E;\hbar)} \sqrt{V_0(y;\hbar)-E}\, dy \\
 T_+(E;\hbar) &:= x_1(E;\hbar)\sqrt{E} - \int_{x_1(E;\hbar)}^\infty \big(\sqrt{E-V_0(y;\hbar)}-\sqrt{E}\big)\, dy \\
 T_-(E;\hbar) &:= -x_2(E;\hbar)\sqrt{E} -  \int_{-\infty}^{x_2(E;\hbar)} \big(\sqrt{E-V_0(y;\hbar)}-\sqrt{E}\big)\, dy
\end{split}\end{equation}
as well as $T(E;\hbar):= T_+(E;\hbar) + T_-(E;\hbar)$.
Then for all $0<\hbar<\hbar_0$ where  $\hbar_0=\hbar_0(V)>0$ is small and $0<E<E_0$
\begin{equation}\label{eq:sentries}
\begin{aligned}
 \Sigma_{11}(E;\hbar) &= e^{-\frac{1}{\hbar}(S(E;\hbar)+iT(E;\hbar))} (1+ \hbar\,\sigma_{11}(E;\hbar)) \\
 \Sigma_{12} (E;\hbar)  &= -i e^{-\frac{2i}{\hbar} T_+(E;\hbar)} (1+\hbar\, \sigma_{12}(E;\hbar))
\end{aligned}
\end{equation}
where the correction terms satisfy the bounds
\begin{equation}
\label{eq:errors}
 |\partial_E^k\, \sigma_{11}(E;\hbar)|+|\partial_E^k\, \sigma_{12}(E;\hbar)| \le C_k\, E^{-k}\quad\forall\; k\ge0,
\end{equation}
with a constant $C_k$ that only depends on $k$ and $V$. The same conclusion holds if instead of~\eqref{eq:V0_def} we
were to define $V_0$ as $V_0:=V+\hbar^2 V_1$ with $V_1\in C^\infty(\Rd)$,
$V_1(x;\hbar)=\frac14\la x\ra^{-2}+O(x^{-3})$ as $x\to\pm\infty$ with $\partial_x^k O(x^{-3}) = O(x^{-3-k})$ for all $k\ge0$
and uniformly in~$0<\hbar\ll1$.
\end{thm}

Note the correction of the original potential by $\frac{\hbar^2}{4}\la x\ra^{-2}$ in~\eqref{eq:V0_def}. Without this correction
the errors $\sigma_{11}$ etc.~diverge as $E\to0$. The proof of this result of course requires a careful analysis of the Jost solutions
which is then needed in the analysis of the stationary phase analysis of~\eqref{eq:semievol}.

The analysis of the Jost solutions is based on the Liouville-Green transform, which we now recall (see~\cite{Olver}).
Given any second order equation $f''(x)=Q(x)f(x)$ on some interval $I$,
and  any diffeomorphism $w:I\to J$ onto some interval $J$, define $g(w):= (w'(x))^{\frac12} f(x)$ where $w=w(x)$. Then by the chain rule,
$f''=Qf$ is the same as $g''(w) = \tilde Q(w) g(w)$ where
\begin{align*}
 \tilde Q(w) &:= \frac{Q(x)}{(w'(x))^2} - (w'(x))^{-\frac32} \del_x^2 (w'(x))^{-\frac12} \\
 &= \frac{Q(x)}{(w'(x))^2} - \frac34 \frac{(w''(x))^2}{(w'(x))^4} + \frac12 \frac{w'''(x)}{(w'(x))^2}
\end{align*}
To apply this transformation, one chooses $w$ so that
\begin{equation}\label{eq:QQ0}
\frac{Q(x)}{(w'(x))^2}=Q_0(w)
\end{equation}
 where $Q_0$ is some {\em normal form}. Then
the problem becomes
\begin{equation}\label{eq:VLG}
\begin{aligned}
g''(w) &= Q_0(w) g(w) - V(w) g(w), \\
 V(w) &:=  \frac34 \frac{(w''(x))^2}{(w'(x))^4} - \frac12 \frac{w'''(x)}{(w'(x))^2}
 \end{aligned}
\end{equation}
where $V$ is treated as a perturbation. This is only admissible if $Q_0$ is in some suitable sense
close to~$Q$. The determination of~$Q_{0}$ is done on a case by case basis.  For example, if $Q$ does not vanish on~$I$, then one can take $Q_0=\sign(Q)$ which leads
to the classical WKB ansatz, i.e.,
\[
Q^{-\frac14}(x) e^{\pm \int^x_{x_0} \sqrt{Q}(y)\, dy} \text{\ \ or\ \ } |Q|^{-\frac14}(x) e^{\pm i\int^x_{x_0} \sqrt{|Q|}(y)\, dy}
\]
depending on whether $Q>0$ or $Q<0$, respectively.
 If $Q$ does vanish at $x_0\in I$ with $Q'(x_0)\ne0$, then one maps $x_0$ to $w=0$ and chooses $Q_0(w)=w$. In other words, the comparison
equation is the Airy equation. The equation for $w$ in that case is $w(x) w'(x)^2 = Q(x)$ which yields
\begin{equation}\label{eq:Langer}
w(x)=\sign(x-x_0) \Big| \frac32 \int_{x_0}^x \sqrt{|Q(y)|}\, dy\Big|^{\frac23}
\end{equation}
which is known as the Langer transform~\cite{Olver}. It is easy to check that  $w$ is (locally around $x_0$) smooth (or analytic) provided  $Q$ is  smooth (or analytic).
It is precisely this Langer transform which is used in~\cite{CSST}, where it is written as follows for $x\ge0$
\begin{align*}
 \zeta&=\zeta(x,E;\hbar)\\
 &:= \sign (x-x_1(E;\hbar)) \Big|\frac32 \int_{x_1(E;\hbar)}^x
\sqrt{|V_0(x;\hbar)-E|}\, d\eta\Big|^{\frac23},
\end{align*}
with $x_1(E;\hbar)>0$ being the unique turning point (for $E$ small). The equation transforms as follows.

\begin{lemma}\label{lem:langer}
There exists $E_0=E_0(V)>0$ so that for all $0<E<E_0$ one has the following properties:
the equation $V_0(x;\hbar)-E=0$ has a unique (simple) solution on $x>0$
which we denote by $x_1=x_1(E;\hbar)$. With $Q_0:=V_{0}-E$
\begin{align}
  \label{eq:zeta}
   \zeta &=\zeta(x,E;\hbar) \\
   &:= \sign (x-x_1(E;\hbar)) \Big|\frac32 \int_{x_1(E;\hbar)}^x
\sqrt{|Q_0(u,E;\hbar)|}\, du\Big|^{\frac23} \nn
\end{align}
defines a smooth change of variables $x\mapsto\zeta$
for all $x\ge0$. Let $q:=-\frac{Q_0}{\zeta}$. Then  $q>0$, $\frac{d\zeta}{dx}=\zeta'=\sqrt{q}$, and
\[
-\hbar^2f''+(V-E)f= 0
\]
transforms into
\begin{equation}\label{eq:Airy}
-\hbar^2 \ddot w(\zeta) = (\zeta+\hbar^2 \tilde V(\zeta,E;\hbar))w(\zeta)
\end{equation}
under $w= \sqrt{\zeta'} f = q^{\frac14} f$. Here $\dot{\
}=\frac{d}{d\zeta}$ and
\[
\tilde V := \frac{1}{4} q^{-1} \langle x\rangle ^{-2} - q^{-\frac14} \frac{d^2
q^{\frac14}}{d\zeta^2}
\]
\end{lemma}

 The asymptotic description of the Jost solutions is found by matching the Airy approximations at
the turning point $w=0$.  A fundamental solution of the transformed equation (i.e., in the $\zeta$ variable)
to the left of the turning point is described in terms of the Airy function $\Ai, \Bi$ by the following result from~\cite{CSST}.

\begin{prop}
  \label{prop:AiryI} Let $\hbar_0>0$ be small. A fundamental system of solutions to~\eqref{eq:Airy} in the range
$\zeta\le0$ is given by
\begin{equation}\nonumber
\begin{aligned}
  \phi_1(\zeta,E,\hbar) &=
  \Ai(\tau) [1+\hbar a_1(\zeta,E,\hbar)] \\
  \phi_2(\zeta,E,\hbar) &=
  \Bi(\tau) [1+\hbar a_2(\zeta,E,\hbar)]
\end{aligned}
\end{equation}
 with
$\tau:=-\hbar^{-\frac23}\zeta$.
 Here $a_1, a_2$ are smooth, real-valued, and they satisfy
the bounds, for all $k\ge0$ and $j=1,2$, and with $\zeta_0:=\zeta(0,E)$,
\begin{align}
  |\partial_E^k  a_j(\zeta,E,\hbar)| &\less  E^{-k} \min\big[\hbar^{\frac13} \la
  \hbar^{-\frac23}
  \zeta\ra^{\frac12}, 1\big] \label{eq:aj_est} \\
 |\partial_E^k \partial_\zeta a_j(\zeta,E,\hbar)| &\les E^{-k} \Big[\hbar^{-\frac13} \la
\hbar^{-\frac23}\zeta\ra^{-\frac12}\chi_{[-1\le \zeta\le0]} \nn \\
&\qquad\qquad +
|\zeta|^{\frac12}  \chi_{[ \zeta_0\le \zeta\le -1]}\Big ] \nn 
\end{align}
uniformly in the parameters $0<\hbar<\hbar_0$, $0<E<E_0$.
\end{prop}

Note that from the standard asymptotic behavior of the Airy functions, viz.,
\[\begin{split}
\Bi(x) &= \pi^{-\frac12} x^{-\frac14} e^{\frac23 x^{\frac32}}
\big[1+O(x^{-\frac32})\big] \text{\ \ as\ \ } x\to\infty
\\\Bi(x)&\ge\Bi(0) >0\quad \forall\;x\ge0 \\
\Ai(x) &= \frac12\pi^{-\frac12} x^{-\frac14} e^{-\frac23
x^{\frac32}} \big[1+O( x^{-\frac32})\big] \text{\ \ as\ \ }
x\to\infty
\\\Ai(x)& >0\quad \forall\;x\ge0
\end{split}\]
the action integral appears naturally in this context, cf.~\eqref{eq:zeta}.
To the right of the turning point one has the following oscillatory basis.

\begin{prop}
  \label{prop:AiryII} Let $\hbar_0>0$ be small. In the range
$\zeta\ge 0$ a basis of solutions to~\eqref{eq:Airy} is given by
\begin{equation}\nonumber
\begin{aligned}
  \psi_1(\zeta,E;\hbar) &=
  (\Ai(\tau)+i\Bi(\tau)) [1+\hbar b_1(\zeta,E;\hbar)] \\
  \psi_2(\zeta,E;\hbar) &=
  (\Ai(\tau)-i\Bi(\tau)) [1+\hbar b_2(\zeta,E;\hbar)]
\end{aligned}
\end{equation}
 with
$\tau:=-\hbar^{-\frac23}\zeta$ and where
$b_1, b_2$  are smooth, complex-valued, and satisfy the bounds
for all $k\ge0$, and $j=1,2$
\begin{equation}\label{eq:bj_est}
\begin{aligned}
 | \partial_E^k\, b_j(\zeta,E;\hbar)| &\le C_{k}\, E^{-k} \la\zeta\ra^{-\frac32} \\
|\partial_\zeta\partial_E^k b_j(\zeta,E)| &\le C_{k}\, E^{-k} \hbar^{-\frac13} \la \hbar^{-\frac23} \zeta \ra^{-\frac12} \la \zeta\ra^{-2}
\end{aligned}
\end{equation}
uniformly in
the parameters $0<\hbar<\hbar_0$, $0<E<E_0$, $\zeta\ge0$.
\end{prop}

We remark that the Langer transform is not the only possibility here. In fact, in~\cite{CDST} an alternative approach is used
which reduces the potential to a Bessel normal-form. This is again done by means of a suitable stretching, i.e., a Liouville-Green transform.

\subsubsection{Intermediate energies and the top of the barrier}
Intermediate energies, including the maximum energy of the potential, can be treated by means of an approximation  of the generalized eigenfunctions. This was carried out 
in detail by Costin, Park, and the author by means of a Liouville Green transformation which reduces the potential near the maximum to a purely quadratic normal form, see~\cite[Proposition 2]{CPS}. In this way, one arrives at a perturbed Weber equation instead of the Airy equation as above.  

However, \cite{DSS2} follows a different route: a Mourre estimate followed by a semi-classical version of the propagation bounds in~\cite{HSS}.
Mourre~\cite{Mourre} introduced the powerful idea that the quantum analog, i.e., \[\chi_I(H)i[H,A]\chi_I(H)\ge \theta\chi_I(H)>0\] where $H=-\Delta+V$,  $A=px+xp$, $p=-i\nabla$
and $\chi_I(H)$ localizes $H$ to some compact interval $I$ of positive energies, entails a limiting absorption bound on the resolvent localized to~$I$ (which is some form
of scattering).
Hunziker, Sigal, Soffer~\cite{HSS}, developed a time-dependent and abstract approach to Mourre theory by means of propagation estimates in the spirit of Sigal, Soffer~\cite{SSof1}.   The main result of~\cite{HSS} is the following theorem.

\begin{thm}
\label{thm:HSS}
Let $A,H$ be self adjoint operators on some Hilbert space and assume the Mourre estimate
\begin{equation}\label{eq:1.6}
E_I i[H,A] E_I \ge \theta E_I
\end{equation}
where $\theta>0$ and $I\subset\Rd$ is some compact interval, and $E_I$ is the spectral projector onto $I$ relative to $H$. Assume, furthermore, that all iterated commutators of $f(H)$ with $A$ are bounded where $f\in C^\infty_0(\Rd)$.
Let $\chi^{\pm}$ be the indicator functions of $\Rd^{\pm}$, respectively. Then for any $m\ge1$,
\[
\| \chi^{-} (A-a-\theta' t) e^{iHt} g(H) \chi^+(A-a) \| \le C(m,\theta,\theta') \, t^{-m}
\]
for any $g\in C_0^\infty(I)$, any $0<\theta'<\theta$, uniformly in $a\in\Rd$.
\end{thm}

As simple consequence of this result is the following propagation estimate, which is clearly most important in the context of Theorem~\ref{thm2}:
\begin{equation}
\label{eq:HSS}
\| \la A\rangle^{-\alpha} e^{iHt} g(H) \la A\rangle^{-\alpha} \| \le C(\alpha) \, \la t\rangle^{-\alpha}
\end{equation}
for any $\alpha>0$.  In application one typically takes $A=\frac12(px+xp)$, the generator of dilations, or some variant thereof. Taking $\alpha=1$ shows
that one needs at least $w_1$ in the Schr\"odinger case of Theorem~\ref{thm2}.

One needs to resolve two issues before applying this theory to Theorem~\ref{thm2}:
\begin{itemize}
\item We require a semi-classical version of~\cite{HSS}.
\item The top of the barrier energy is trapping in the classical sense.
\end{itemize}
While the first issue is a routine variant of~\cite{HSS}, the second is not. In the nontrapping case, Graf~\cite{Graf} and Hislop, Nakamura~\cite{HN} showed that
the classical nontrapping condition $\{a,h\}>\alpha>0$ on the entire energy level $\{h=E_0>0\}$ implies the Mourre estimate~\eqref{eq:1.6} for $I$ some small
interval around~$E_0$ (in the semi-classical case with $\hbar$ sufficiently small).
In the case of surfaces of revolution as in Theorem~\ref{thm2} this fact, together with Theorem~\ref{thm:HSS}, implies that one can handle energies in the range $\epsilon<E<
V_{\rm max}(0)-\epsilon$ since they verify a classical nontrapping condition. On the other hand, for energies near $V_{\rm max}(0)$ this fails since the top energy is
classically trapping. Nevertheless, the Heisenberg uncertainty principle (or the semiclassical harmonic oscillator) guarantee~\eqref{eq:1.6}.

 Indeed, with $V(x)=1- \frac12\la Qx,x\ran + O(|x|^3)$ with $Q$ positive definite,
 \[
 \{ h,a\}=\xi^2 - x\cdot\nabla V = \xi^2 + \la Q x,x\ran + O(|x|^3) \ge \theta(\xi^2 + x^2)
  \]
for small $x$. However, $p^2+q^2\ge c>0$ by the uncertainty principle, which indicates that one should expect that~\eqref{eq:1.6}
continues to hold at a non-degenerate maximum. For a rigorous rendition of this argument see Briet, Combes, Duclos~\cite{BCD}, Nakamura~\cite{Nak}, and~\cite{DSS2}.

Generally speaking, the problem of obtaining a  representation of the resolvent and the spectral measure, and of proving a limiting absorption principle
for energies near a potential barrier has received a lot of attention, see the monograph by Bony et al.~\cite{Bony} and the earlier literature cited there such 
as the classical  work by Helffer, Sj\"ostrand from the 80s.

This concludes our informal sketch of the proof of~\eqref{eq:schr_d2}. As for~\eqref{eq:schr_d3}, one proceeds analogously by dividing energies into three
regions, low, intermediate, and high. In the low and high cases, one obtains pointwise bounds without weights from the WKB arguments outlined above, followed
by oscillatory integral estimates as in~\cite{SSS2}. For the intermediate regime one uses the $L^2$ bound (from the Mourre-Hunziker-Sigal-Soffer estimates)
which requires a weight~$w_1$ followed by the Sobolev embedding theorem. Note that the latter costs one power of~$\ell$, whereas summation over~$\ell$ requires
another weight of the form $\ell^{1+\epsilon}$ which explains the loss of $(1-\del^2_\theta)^{1+\epsilon}$ on the right-hand side of~\eqref{eq:schr_d3} as compared
to~\eqref{eq:schr_d2}.

As a final remark, we would like to emphasize that the sketch of proof of Theorem~\ref{thm2} which we just concluded is an adaptation of the
argument which was developed for the Schwarzschild case in~\cite{DSS2}.

\subsection{The Schwarzschild case}

The results on surfaces of revolution are relevant to another problem namely the decay of linear waves
on a Schwarzschild black hole background. To be more specific,
 choose coordinates such that the
exterior region of
the black hole can be written as $(t,r,(\theta,\phi)) \in \mathbb{R} \times
( 2M, \infty) \times S^2$ with the metric
$$
g = -F(r)dt^2 + F(r)^{-1} dr^2 + r^2(d\theta ^2 +  \sin^2\theta d\phi^2)
$$
where $ F(r) = 1 - \frac{2M}{r}$ and, as usual, $M>0$ denotes the mass.
We now introduce the well--known \emph{Regge--Wheeler
tortoise coordinate} $r_*$ which (up to an additive constant)
is defined by the relation
$$
F = \frac{dr}{dr_*}.
$$
In this new coordinate system, the outer region is described by
$(t,r_*,(\theta,\phi)) \in \mathbb{R} \times \mathbb{R} \times S^2$,
\begin{equation}
\label{eq_sstortoise}
g = -F(r)dt^2 +F(r)dr_*^2 +r^2(d\theta^2 + \sin^2 \theta d\phi ^2)
\end{equation}
with $F$ as above and $r$ is now interpreted as a function of $r_*$.
Explicitly,  $r_*$ is computed as
\begin{equation*}
r_* = r + 2M\log \left(\frac{r}{2M}-1\right).
\end{equation*}
Generally, the Laplace--Beltrami operator on a manifold with metric $g$
is given by
$$ \Box_g=\frac{1}{\sqrt{|\det (g_{\mu \nu})|}}\partial_\mu
\left (\sqrt{|\det (g_{\mu \nu})|}g^{\mu \nu}\partial_\nu \right ) $$
and thus, for the metric $g$ in (\ref{eq_sstortoise}), we obtain
$$ \Box_g=F^{-1}\left (-\partial_t^2+\frac{1}{r^2}\partial_{r_*} \left (r^2
\partial_{r_*} \right ) \right )+\frac{1}{r^2}\Delta_{S^2}. $$
By setting $\psi(t,r_*,\theta, \phi)=r(r_*)\tilde{\psi}(t,r_*,\theta,\phi)$
and writing $x=r_*$,
the wave equation $\Box_g
\tilde{\psi}=0$ is equivalent to
\begin{equation}
\label{eq_wavess}
-\partial_t^2 \psi+\partial_x^2
\psi-\frac{F}{r}\frac{dF}{dr}\psi+\frac{F}{r^2}\Delta_{S^2}\psi=0.
\end{equation}

The mathematically rigorous analysis of this equation goes back to
Wald~\cite{Wal} and Kay~\cite{KW}, who established uniform boundedness of solutions.
In the spirit of the positive commutator methods outlined above, Dafermos and Rodniansk~\cite{DR3} found a robust approach based on carefully chosen vector fields and multipliers. See Luk's work~\cite{Luk1, Luk2} which is in a similar spirit.  As already noted,  Blue and Soffer~\cite{BSof}  proved  local decay estimates using Morawetz estimates.  Dafermos and Rodnianski~\cite{DR2} proved Price's $t^{-3}$ decay law for a nonlinear problem  but assuming spherical symmetry.

The purpose of this section is to discuss recent work of Donninger and the authors on pointwise decay for
solutions to~Eq.~(\ref{eq_wavess}).  Different types of decay
estimates have been proved before.
Our results differ from the above in certain respects: the methods
we use are based on constructing the Green's function and deriving
the needed estimates on it. Previous works in this direction include mainly the series of
papers by Finster, Kamran, Smoller and Yau, see for example~\cite{FKSY},  where the first pointwise decay result
for Kerr black holes was proved.

As in the case for surfaces of revolution,  we freeze the
angular momentum $\ell$ or, in other words, we project onto a
spherical harmonic.
More precisely, let $Y_{\ell,m}$ be a spherical harmonic (that is, an eigenfunction of
the Laplacian on $S^2$ with eigenvalue $-\ell(\ell+1)$) and insert the
ansatz $\psi(t,x,\theta,\phi)
=\psi_{\ell,m}(t,x)Y_{\ell,m}(\theta,\phi)$ in Eq.~(\ref{eq_wavess}).
This yields the \emph{Regge--Wheeler equation}
$$
\partial_t^2 \psi_{\ell,m}-\partial_x^2
\psi_{\ell,m}+V_{\ell,\sigma}(x)\psi_{\ell,m}=0
$$
with $\sigma=1$
where
$$ V_{\ell,\sigma}(x)=\left (1-\frac{2M}{r(x)} \right )
\left (\frac{\ell(\ell+1)}{r^2(x)}+\frac{2M\sigma}{r^3(x)} \right ) $$
is known as the \emph{Regge--Wheeler potential}.
The other physically relevant values of the parameter~$\sigma$ are $\sigma=-3,0$.
For more background we refer the reader to the introduction of~\cite{DSS1}, or~\cite{DR2}.

We immediately note some crucial features of~$V_{\ell,\sigma}$: it decays exponentially as $x\to-\infty$,
it decays according to an inverse square law as $x\to+\infty$ provided $\ell>0$, and like an inverse cube if $\ell=0$.
Moreover, it has a unique nondegenerate maximum which is located at the  {\em photon sphere}. It consists of closed
light rays and replaces the unique periodic geodesic which we encountered in Theorem~\ref{thm2}.

So we expect that at least some of the machinery that we described
above in the surface case applies here as well. However, the
Regge-Wheeller potential is considerably more difficult to deal
with.

The main result of~\cite{DSS1} is the following pointwise decay, which captures the so-called Price law for fixed
angular momentum.  Strictly speaking, it is still off by one power of~$t$ from the sharpest
form of Price's law which is $t^{-2\ell-3}$ whereas the following result proves $t^{-2\ell-2}$
(we shall comment on that issue below). Note how the accelerated decay for higher values of~$\ell$ mirrors
what we saw for the surfaces of revolution in Theorem~\ref{thm1}. Hintz~\cite{Hintz} recently closed the gap of the missing power of~$t$ and thus finished the proof of Price's law. 

\begin{thm}
\label{thm:DSS1}
Let $(\ell,\sigma) \notin \{(0,0), (0,-3), (1,-3)\}$, $\alpha \in \mathbb{N}$
and $1 \leq \alpha \leq 2\ell+3$.
Then the solution operators for the Regge--Wheeler equation satisfy the
estimates
\begin{align*}
 \|w_\alpha \cos(t\sqrt{\mc{H}_{\ell,\sigma}})f\|_{L^\infty(\mathbb{R})}
&\leq C_{\ell,\alpha} \langle t \rangle^{-\alpha}\Big(\Big \|\frac{f'}{w_\alpha}\Big
\|_{L^1(\mathbb{R})}\\
&\qquad\qquad + \Big\|\frac{f}{w_\alpha}\Big
\|_{L^1(\mathbb{R})} \Big) 
\end{align*}
and
$$ \left \|w_\alpha
\frac{\sin(t\sqrt{\mc{H}_{\ell,\sigma}})}{\sqrt{\mc{H}_{\ell,\sigma}}}f \right
\|_{L^\infty(\mathbb{R})}
\leq C_{\ell,\alpha} \langle t \rangle^{-\alpha+1}\left \|\frac{f}{w_\alpha}\right
\|_{L^1(\mathbb{R})} $$
for all $t \geq 0$ where $w_\alpha(x):=\langle x \rangle^{-\alpha}$.
\end{thm}

The values of $(\sigma,\ell)$ which we exclude here are precisely those where the Regge-Wheeler potential
gives rise to zero energy resonances. Physically speaking, they correspond to a gauge invariance, such as changing the mass, and are therefore
irrelevant.

The proof of Theorem~\ref{thm:DSS1} is based on
representing the solution as an oscillatory
integral in the energy variable $\lambda$, schematically one may write
$$
\psi(t,x) = \int U(t,\lambda) \mathrm{Im} \left [
G_{\ell,\sigma}(x,x',\lambda) \right ] f(x')\,
dx'd\lambda
$$
where  $U(t,\lambda)$ is a combination of $\cos (t\lambda) $ and
$\sin (t\lambda)$ terms and $G_{\ell,\sigma}(x,x',\lambda)$
is the kernel (Green's function) of the
resolvent of the  operator $\mc{H}_{\ell,\sigma}$.  In analogy with Theorem~\ref{thm1},
$G_{\ell,\sigma}(x,x',\lambda)$ is constructed in terms of the
Jost solutions and we obtain these functions in various domains of
the $(x,\lambda) $ plane by perturbative arguments: for $|x \lambda|$ small
we perturb in $\lambda$ around $\lambda=0$, whereas for $|x \lambda|$ large we
perturb off of Hankel functions.
This is done in such a way that there remains a small window where the two
different
perturbative solutions can be glued together.
One of the main technical difficulties of the proof lies with the fact that we
need good estimates for arbitrary derivatives of the perturbative solutions.
This is necessary in order to control the oscillatory integrals.
The most important contributions come from $\lambda \sim 0$ and we therefore
need to derive the
exact asymptotics of the Green's function and its derivatives in the limit
$\lambda \to 0$.
For instance, we prove that
$$ \mathrm{Im}\left [G_{\ell,\sigma}(0,0,\lambda) \right
]=\lambda P_\ell(\lambda^2)+O(\lambda^{2\ell+1}) $$
as $\lambda \to 0+$ where $P_\ell$ is a polynomial of degree $\ell-1$ (we set
$P_0 \equiv 0$) and the
$O$--term satisfies $O^{(k)}(\lambda^{2\ell+1})=O(\lambda^{2\ell+1-k})$ for all
$k \in \mathbb{N}_0$.

As already noted before, for $\ell=0$ the Regge-Wheeler potential decays like an inverse cube as $x\to\infty$. This case
is covered by the following result of
Donninger and the first author~\cite{DS}.

\begin{thm}
\label{thm:DS}
Let $V \in C^{[\alpha]+1}(\mathbb{R})$ with
$V(x)=|x|^{-\alpha}[c_\pm+O(|x|^{-\beta})]$ as $x \to \pm \infty$ where
$2<\alpha \leq 4$, $\beta=\frac{1}{2}(\alpha-2)^2$,
$c_\pm \in \mathbb{R}$ and $|O^{(k)}(|x|^{-\beta})|\lesssim |x|^{-\beta-k}$ for
$k=1,2,\dots,[\alpha]+1$. Denote by $A$ the self--adjoint Schr\"odinger operator
$Af:=-f''+Vf$ in $L^2(\mathbb{R})$
and assume that $A$ has no bound states and no resonance at zero
energy.
Then the following decay bounds hold:
\begin{align*} 
\|\langle \cdot \rangle^{-\alpha-1}
\cos(t\sqrt{A})f\|_{L^\infty(\mathbb{R})}
& \lesssim \langle t \rangle^{-\alpha} \big (\|\langle \cdot
\rangle^{\alpha+1}f'\|_{L^1(\mathbb{R})} \\
&\qquad +
\|\langle \cdot
\rangle^{\alpha+1}f\|_{L^1(\mathbb{R})}\big ) 
\end{align*}
and
$$ \left \|\langle \cdot \rangle^{-\alpha-1} \frac{\sin(t\sqrt{A})}{\sqrt{A}}
f \right \|_{L^\infty(\mathbb{R})}
\lesssim \langle t \rangle^{-\alpha} \|\langle \cdot
\rangle^{\alpha+1}f\|_{L^1(\mathbb{R})} $$
for all $t \geq 0$.
\end{thm}

In particular, this gives $t^{-3}$ for $\alpha=3$ which is the sharp form of Price's law for $\ell=0$.
It is important to realize that the decay of the waves in Theorems~\ref{thm:DSS1} and~\ref{thm:DS} is really a manifestation of transport
rather than of dispersion.  Indeed, d'Alembert's formula shows that any solution of
\[
\partial_{tt} u - \partial_{xx} u =0,\qquad u(0)=f,\; \partial_t u(0)=g
\]
with Schwartz data (say) and $\int g(x)\,dx=0$ satisfies
\[
\| \la x\rangle^{-\alpha} u(t)\|_\infty \le C(\alpha)\, t^{-\alpha}
\]
for {\em any} $\alpha\ge0$.
This vanishing mean condition can be attributed to the zero energy resonance for the free Laplacian in one dimension.
 Needless to say, the one-dimensional problem does not exhibit any sort of  dispersion but is governed by linear transport
which leads to this arbitrary local decay of the waves.
It is very interesting to note (but perhaps not immediately clear) that the sharp Huyghens principle in three dimensions is still visible in the local
decay law of Theorem~\ref{thm:DSS1}. In fact, we claim that the sharp $t^{-2\ell-3}$ Price law (at least for $\ell\ge1$) is a result of the correction term
of the form~$\frac{\log x}{x^3}$ in the Regge-Wheeler potential rather than the leading inverse square decay as $x\to+\infty$.

To clarify this point, we now present a simple model case from~\cite{CDST}.
With $a>0$,
\[
\calH:=-\del_x^2 + V, \quad V(x)=\left\{ \begin{array}{ll} 0 & \text{\ if\ } x\le -1\\
\frac{a^2-\frac14}{x^2} & \text{\ if\ }x\ge 1 \end{array}\right.
\]
Moreover, $V\in C^\infty(\Rd)$ is such that $\calH$ has {\em no zero energy resonance} which means
that there does not exist a globally subordinate (or recessive)
 solution $\calH f=0$ other than $f\equiv 0$.  Recall that this refers to solutions of the slowest
allowed growth at both ends, which means here that $f(x)=O(1)$ as $x\to-\infty$ and $f(x)=O(x^{\frac12-a})$ as $x\to+\infty$. 
Then one has the following local decay estimates
 for the wave equation with potential~$V$.

\begin{prop}\label{prop:modeldecay}
Under the above assumptions on~$\calH$,
\begin{align*}
\Big\| \la x\ra^{-\sigma} \frac{\sin(t\sqrt{\calH})}{\sqrt{\calH}}P_{(0,\infty)}(\calH) g\Big\|_\infty &\le C\la t\ra^{-2a-1} \big\| \la x\ra^\sigma g \big\|_1 \\
\Big\| \la x\ra^{-\sigma}  \cos(t\sqrt{\calH}) P_{(0,\infty)}(\calH) f\Big\|_\infty &\le C\la t\ra^{-2a-2} ( \big\| \la x\ra^\sigma f \big\|_1 \\
&\qquad +  \big\| \la x\ra^\sigma f' \big\|_1)
\end{align*}
where $\sigma>0$ is sufficiently large depending on~$a$. These decay rates are optimal provided $a\not\in \Z_0^++\frac12$. In the latter case, one obtains
decay $t^{-N}$ for any $N$ (provided $\sigma$ is taken sufficiently large depending on~$N$).
\end{prop}
\begin{proof}
We prove the first bound, the second one being very similar. Thus, let $\psi(t,x)$ be a solution of the problem
\[
\del_t^2 \psi-\del_x^2\psi+V\psi=0,\quad \psi(0,x)=0, \; \del_t\psi(0,x)=g
\]
where $g$ is Schwartz, say,  and set for $\Re(p)>0$
\[
\hat \psi(p,x) := \int_0^\infty e^{-tp} \psi(t,x)\, dt
\]
Then
\[
(\calH+p^2)\hat{\psi}(p,\cdot)=g
\]
which has a unique bounded solution
\begin{align*}
\hat{\psi}(p,x) &=  \int_{-\infty}^\infty G(p;x,y)\, g(y)\, dy \\
&= \int_{-\infty}^x \frac{f_+(x,p)f_-(y,p)}{W(p)}\, g(y)\, dy \\
&\qquad - \int_x^\infty \frac{f_+(y,p)f_-(x,p)}{W(p)}\, g(y)\, dy \\
\end{align*}
with constant Wronskian $W(p):= f_+(x,p)f_-'(x,p)-f_+'(x,p) f_-(x,p)$. Here $f_{\pm}(x,p)$ are the Jost solutions
\[
(\calH+p^2) f_{\pm}(\cdot,p)=0, \qquad f_{\pm}(x,p)\sim e^{\mp xp} \text{\ as\ }x\to \pm\infty
\]
The goal is now to obtain the expansion of $f_{\pm}(x,p)$ in small $p$, as this then yields the large time asymptotics of,
with arbitrary $p_0>0$,
\begin{equation}\label{eq:psiint}
\psi(t,x) =\frac{1}{2\pi i} \int_{-\infty}^\infty \int_{p_0-i\infty}^{p_0+i\infty} e^{tp}\;  G(p;x,y) \,dp\; g(y)\, dy
\end{equation}
via contour deformation and Watson's lemma.  By choice of potential $V$,
\begin{align*}
f_-(x,p) &= e^{px} \text{\ for\ }x\le -1\\
f_+(x,p) &= \frac{\pi i}{2} e^{a\pi i/2} H_a^{(1)}(ipx) \big(\frac{2px}{\pi}\big)^{\frac12} \text{\ for\ }x\ge 1
\end{align*}
One can continue $f_-(x,p)$ to the right of~$x=-1$ which yields an entire function in~$p$ for each fixed~$x$.
The nonresonance condition for $p=0$  means that $f_-(\cdot,0)$ and $x^{\frac12-a}$ are linearly independent at $x=1$.
Since $H_a^{(1)}=J_a+iY_a$ and -- up to constant factors --
\[
J_a(u)\sim u^a (1+O(u^2)),\quad Y_a(u) \sim u^{-a}(1+O(u^2))
\]
as $u\to0$ with analytic $O(u^2)$ (at least provided $a$ is not an integer), we conclude that
\[
W(p) = c(V)\, p^{\frac12-a} \big[ 1+O(p^2)+\tilde c(V)\, p^{2a}(1+O(p^2)) \big] \text{\ \ as\ }p\to0
\]
with $O(p^2)$ analytic in a neighborhood of $p=0$ and with $c(V)\ne0$.  This is obtained by computing $W(p)$ at $x=1$, say,
and by noting that  the most singular contribution to~$W(p)$ around $p=0$ is $c(V)\, p^{\frac12-a}$. By inspection, $c(V)=0$ is
the same as a zero energy resonance which is excluded.
If $a$ is a positive integer, then as $p\to0$
\[
W(p) = c(V)\, p^{\frac12-a} \big[ 1+O(p^2)+\tilde c(V)\, p^{2a}\log(p)(1+O(p^2)) \big] 
\]
For simplicity,
let us first freeze $x,y$, say $x=y=1$. Then one concludes from the preceding that
\begin{equation}\label{eq:p2a}
G(p;1,1)= C(V)\, p^{2a} \big[ 1+O(p^2)+\tilde c(V)\, p^{2a}(1+O(p^2)) \big]
\end{equation}
for small $p\in \C\setminus(-\infty,0]$, and analytic $O(p^2)$ around $p=0$, whereas for the case of $a\in\Z$,
\begin{align*}
G(p;1,1) &= C(V)\, p^{2a} \log(p) \big[ 1+O(p^2)\\
&\qquad +\tilde c(V)\, p^{2a} \log(p) (1+O(p^2)) \big]
\end{align*}
The stated decay law now follows  via Watson's lemma in a standard fashion. Note the special role of integer but odd $2a$ (which is the exceptional
case in the statement of the proposition): in that case
\eqref{eq:p2a} is analytic in small~$p$ whence one can push the contour in~\eqref{eq:psiint} through $p=0$ leading to exponential decay (at least as
far as the contribution of small~$p$ is concerned).

\noindent We now discuss the Watson lemma in more detail. First, we move the contour in~\eqref{eq:psiint} onto the imaginary axis:
\[
\psi(t,1) =\frac{1}{2\pi i}   \int_{ -i\infty}^{ i\infty} e^{tp}\;  G(p;1,1) \,dp
\]
The contribution due to $1-\chi(p)$ is shown via integration by parts to decay faster than any power of~$t$ (use that $G(iE;x,y) =   O(E^{-1})$
for large $E$, uniformly in $x,y$). On the other hand, for the contribution of~$\chi$
we retain only finitely many terms from~$G(p;1,1)$ with a remainder that is smooth enough around $p=0$ so as to yield the desired decay again by integration by parts.
Finally, the first remaining term is of the form (up to a constant factor $C(V)$)
\[
  \int_{-i\eps}^{i\eps} e^{tp}\;  p^{2a} \chi(p) \,dp
\]
We also have a $\log p$ factor if $a\in\Z$.  One now extends this to
\begin{equation}\label{eq:babyprice}
  \int_\gamma  e^{tp}\;  p^{2a}   \,dp
\end{equation}
where $\gamma$ is an curve which contains $[-i\eps,i\eps]$ and is asymptotic to $[0,e^{i\theta}\infty]$ and $[-e^{i\theta}\infty,0]$, respectively, and the ends.
Noting that the integrals we inserted here decrease like~$t^{-N}$  for any $N$ by integration by parts.
By Cauchy's theorem this is the same as
\[
2\sin(2a\pi) \int_0^\infty   e^{-tp}\;  p^{2a}  \,dp  = 2\sin(2a\pi) t^{-2a-1} \Gamma(2a+1)
\]
which is the decay rate stated in the proposition.
Note that if $a=\ell+\frac12$ then this term vanishes leading to the exceptional behavior stated above. On the other hand, if $a\in\Z$,
then this contribution does not vanish due to the $\log(p)$ factor.
Finally, we need to remove the restriction $x=y=1$. However, we have set up our argument in such a way that this modification is easy. First,
the contribution of $|p|>\eps$ is again shown to decay at an arbitrary rate via integration by parts.  Now this procedure brings down as many powers of $x,y$ as given by the
desired power of~$t^{-1}$.  Next, the contribution of the finitely many terms involving $p^{2a}$ etc.~is similar to before, and each one of these terms comes with a
corresponding weight in $x$ and~$y$. Finally, the remainder in~$G(p;x,y)$ after subtracting that initial segment is again sufficiently smooth in~$p$ and therefore
integration by parts yields the desired decay leading to another instance of requiring large~$\sigma$.
\end{proof}

The significance of this proposition lies with proximity of $V$ to the Regge-Wheeler potential. Indeed, we replaced the exponential tails on the left by zero,
and retained the inverse square tails on the right (ignoring the higher-order corrections). In case of the Regge-Wheeler potential one has  $a^2-\frac14=\ell(\ell+1)$ which implies that $a=\ell+\frac12$ which
is the {\em exceptional case of Proposition~\ref{prop:modeldecay}}. Formally speaking $2a+1=2\ell+2$ corresponds exactly to the decay rate of Theorem~\ref{thm:DSS1}, whereas the Price law $t^{-2\ell-3}$ is therefore seen to be a result of the $\frac{\log x}{x^3}$
correction to the far field in $V_{\ell,\sigma}$. In fact, it is shown in~\cite{CDST} that  the Price law is due to the nonanalytic term $p^{2a+1}\log p$ instead of~$p^{2a}$ in~\eqref{eq:babyprice}. To accomplish this, one derives an expansion of $f_+(x,p)$ in small~$p$ taking into account as many terms from~$V_{\ell,\sigma}$
as required for obtaining Price's law and the next few corrections to it. The route taken in~\cite{CDST} consists of a reduction of the Regge-Wheeler
potential to a normal form by means of a Liouville-Green transform. The normal form here consists of the potential without any corrections to the leading $\frac{\ell(\ell+1)}{x^2}$ decay. The branching around $p=0$ then results from the change of independent variable.  Arguing as in the previous proof then yields
the sharp $t^{-2\ell-3}$ Price law.

To conclude this survey, let us state the main local decay result from~\cite{DSS2}.

\begin{thm}
 \label{thm:DSS2}
The following decay estimates hold for solutions $\psi$ of~\eqref{eq_wavess} with data $\psi[0]=(\psi_0,\psi_1)$:
\begin{align}
 \| \la x\ra^{-\frac92-}  \psi(t) \|_{L^2} &\les \la t\ra^{-3} \| \la x\ra^{\frac92+} (\snabla^5 \del_x\psi_0, \snabla^5 \psi_0, \snabla^4\psi_1) \|_{L^2} \label{eq:decaywaveL2}  \\
\| \la x\ra^{-4}  \psi(t) \|_{L^\infty} &\les \la t\ra^{-3} \| \la x\ra^4 (\snabla^{10} \del_x\psi_0, \snabla^{10} \psi_0, \snabla^{9} \psi_1) \|_{L^1} \label{eq:decaywaveL1}
\end{align}
where $\snabla$ stands for the angular derivatives. The notation $a\pm$ stands for $a\pm\eps$ where $\eps>0$ is arbitrary (the choice determines the constants involved). Also, instead
of $(\snabla^{10},\snabla^9)$ in~\eqref{eq:decaywaveL1} one needs less, namely $(\snabla^{\sigma+1},\snabla^{\sigma})$ where $\sigma>8$ is arbitrary. Here $L^2:=L^2_x(\Rd;L^2(S^2))$, $L^1:=L^1_x(\Rd;L^1(S^2))$, and $L^\infty:=L^\infty_x(\Rd;L^\infty(S^2))$.
\end{thm}

It is obtained by summation in~$\ell$ following the same line of reasoning that lead to Theorem~\ref{thm2} above. The most significant complication is due to
the asymmetry of the Regge-Wheeler potential: while the inverse square potential for $x\to\infty$ is covered by~\cite{CSST} as before, the exponentially decaying part
on the left requires another WKB analysis. We refer the reader to~\cite{DSS2} for the details.

\begin{acknowledgments}
This article is dedicated to the memory of Jean Bourgain. The author was partly supported by the National
Science Foundation grant DMS-1902691. The author thanks an anonymous referee for numerous helpful comments which improved the presentation. 
\end{acknowledgments}


\section*{Data Availability Statement}  Data sharing is not applicable to this article as no new data were created or analyzed in this study.

\end{document}